\documentclass[reqno, 11pt]{amsart}

\usepackage{url}
\usepackage[colorlinks]{hyperref}

\usepackage{helvet}

\newtheorem{theorem}{Theorem}[section]
\newtheorem{lemma}[theorem]{Lemma}
\newtheorem{slemma}[theorem]{Sublemma}
\newtheorem{proposition}[theorem]{Proposition}
\newtheorem{corollary}[theorem]{Corollary}

\theoremstyle{definition}
\newtheorem{definition}[theorem]{Definition}

\newtheorem{remark}[theorem]{Remark}

\numberwithin{equation}{section}

\usepackage{bbm}
\usepackage{euscript}
\usepackage{amsmath}
\usepackage{amsthm}

\usepackage{mathrsfs}
\usepackage{amsxtra}
\usepackage{amssymb}
\usepackage{pifont}
\usepackage{amsbsy}

\usepackage{graphicx}
\usepackage{epstopdf}

\newskip\aline \newskip\halfaline
\def\skipaline{\vskip\aline}

\def\qedbox{$\rlap{$\sqcap$}\sqcup$}
\def\qed{\nobreak\hfill\penalty250 \hbox{}\nobreak\hfill\qedbox\skipaline}

\def\proofend{\eqno{\mbox{\qedbox}}}

\newcommand{\one}{\mathbbm{1}}

\newcommand\bC{{\mathbb C}}

\newcommand\bR{{\mathbb R}}

\newcommand\bZ{{\mathbb Z}}

\DeclareMathOperator{\uso}{\underline{\mathit{so}}}

\DeclareMathOperator{\Gr}{\mathbf{Gr}}
\DeclareMathOperator{\diag}{Diag} \DeclareMathOperator{\Hom}{Hom}
 \DeclareMathOperator{\End}{End}

\DeclareMathOperator{\spa}{span}

\DeclareMathOperator{\ev}{\mathbf{ev}}

\DeclareMathOperator{\pf}{\mathbf{Pf}}

\DeclareMathOperator{\var}{\boldsymbol{var}}

\DeclareMathOperator{\spec}{spec}

\newcommand{\be}{{\boldsymbol{e}}}

\newcommand{\ii}{\boldsymbol{i}}

\newcommand{\bu}{{\boldsymbol{u}}}

\newcommand{\bv}{{\boldsymbol{v}}}

\newcommand{\bx}{{\boldsymbol{x}}}
\newcommand{\by}{{\boldsymbol{y}}}

\newcommand{\bsC}{\boldsymbol{C}}
\newcommand{\bsD}{\boldsymbol{D}}
\newcommand{\bsE}{{\boldsymbol{E}}}

\newcommand{\bsK}{{\boldsymbol{K}}}

\newcommand{\bsP}{\boldsymbol{P}}

\newcommand{\bsU}{{\boldsymbol{U}}}
\newcommand{\bsV}{\boldsymbol{V}}

\newcommand{\bgamma}{\boldsymbol{\gamma}}

\newcommand{\bchi}{\boldsymbol{\chi}}

\newcommand{\bsi}{\boldsymbol{\sigma}}

\newcommand{\si}{{\sigma}}

\newcommand{\ve}{{\varepsilon}}
\newcommand{\eps}{{\epsilon}}
\newcommand{\vfi}{{\varphi}}

\newcommand{\fa}{\mathfrak{a}}

\newcommand{\eH}{\EuScript H}

\newcommand{\eK}{\EuScript{K}}
\newcommand{\eL}{\EuScript{L}}

\newcommand{\eO}{\EuScript{O}}

\newcommand{\eS}{\EuScript{S}}
\newcommand{\eT}{\EuScript{T}}

\newcommand{\eX}{\EuScript{X}}

\newcommand{\ra}{\rightarrow}

\newcommand{\Llra}{{\Longleftrightarrow}}

\newcommand{\lan}{\langle}
\newcommand{\ran}{\rangle}

\def\inpr{\mathbin{\hbox to 6pt{\vrule height0.4pt width5pt depth0pt \kern-.4pt \vrule height6pt width0.4pt depth0pt\hss}}}

\newcommand{\pa}{\partial}

\newcommand{\nah}{\widehat{\nabla}}

\newcommand{\Ra}{\Rightarrow}

\begin{document}

\title[Gauss-Bonnet-Chern Theorem]{The Gauss-Bonnet-Chern theorem: a probabilistic perspective} 

% This is the American J. of Math Version

\date{Started  March 5, 2014. Completed  on March 23, 2014. Last modified on {\today}. }

\author{Liviu I. Nicolaescu}
%\thanks{This work was partially supported by the NSF grant, DMS-1005745.}

\address{Department of Mathematics, University of Notre Dame, Notre Dame, IN 46556-4618.}
\email{nicolaescu.1@nd.edu}
\urladdr{\url{http://www.nd.edu/~lnicolae/}}
\author{Nikhil Savale}
\address{Department of Mathematics, University of Notre Dame, Notre Dame, IN 46556-4618.}
\email{nsavale@nd.edu}

\begin{abstract} We prove that the Euler form of a metric connection on  a real oriented vector bundle  $E$  over a compact oriented manifold $M$   can be identified, as a  current, with the expectation of the random current defined by the  zero-locus of a certain random section of the bundle. We also  explain  how to reconstruct probabilistically  the  metric and the connection on $E$ from the statistics  of random sections of $E$.
\end{abstract}

\maketitle

\tableofcontents

\section{Introduction}
\setcounter{equation}{0}

 \subsection{The Gauss-Bonnet-Chern theorem} We begin by recalling the classical  Gauss-Bonnet-Chern theorem \cite{Ch, N1, Spi5}. Suppose that $E\to M$ is a real \emph{oriented}  vector bundle of even rank $r=2h$ over the smooth, compact oriented manifold $M$ of dimension $m$.  Fix a metric $(-,-)_E$ on $E$ and a connection   $\nabla^E$ compatible with the metric. We denote by $F^E$ the curvature of the connection $\nabla^E$ on $E$. The \emph{Euler form} of $(E,\nabla^E)$ is the   closed form
\begin{equation}
\be(E,\nabla^E):=\frac{1}{ (2\pi)^{h} } \pf\bigl(-F^E\bigr)\in\Omega^r(M),\;\;r=2h,
\label{euler}
\end{equation}
where $\pf$  denotes the Pfaffian construction, \cite{BZ,MQ, N1}.    For the applications we have in mind  it is important to have  an explicit  local description $\pf\bigl(-F^E\bigr)$.  

If we fix a local, \emph{positively oriented}  orthonormal frame $\be_1,\dotsc,\be_r$ of $E$ defined on some open set $\eO\subset M$, then the curvature $F^E$ is represented by a skew-symmetric  $r\times r$ matrix 
\[
F^E=(F_{\alpha\beta}^E)_{1\leq \alpha,\beta\leq r},\;\;F_{\alpha\beta}\in \Omega^2(\eO).
\]
If we denote by $\eS_r$ the group of permutations of $\{1,\dotsc, r=2h\}$, then (see \cite[\S 8.1.4]{N1})
\begin{equation}
\pf\bigl(-F^E\bigr)=\frac{1}{2^h h!}\sum_{\si\in\eS_r}\eps(\si) F^E_{\si_1\si_2}\wedge \cdots \wedge F^E_{\si_{2h-1}\si_{2h}}\in\Omega^{2h}(\eO),
\label{pf1}
\end{equation}
where $\eps(\si)$ denotes the signature of the permutation $\si\in\eS_r$.

Suppose additionally that  we have  local coordinates $(x^1,\dotsc, x^m)$ on $\eO$.  For $1\leq \alpha_1,\alpha_2\leq r$ and $1\leq j_1,j_2\leq m$ we set
\begin{equation}
F^E_{\alpha_1\alpha_2|j_1j_2}:= F^E_{i_1i_2}(\pa_{x^{j_1}},\pa_{x^{j_2}}).
\label{not}
\end{equation}
Denote  by $\eS_r'$ the subset of $\eS_r$ consisting  of permutations  $(\si_1,\dotsc,\si_{2h})$ such that
\[
\si_1<\si_2,\;\si_3<\si_4,\;\dotsc ,\;\si_{2h-1}<\si_{2h}.
\]
We deduce from (\ref{pf1}) that 
\begin{equation}
\pf\bigl(-F^E\bigr)\bigl(\,\pa_{x^1},\cdots,\pa_{x^r}\,\bigr)= \frac{1}{h!}\sum_{\vfi,\si\in\eS_r'} \eps(\si\vfi)F^E_{\si_1\si_2|\vfi_1\vfi_2}\cdots F^E_{\si_{2h-1}\si_{2h}|\vfi_{2h-1}\vfi_{2h}}.
\label{pf2}
\end{equation}

We denote by $\Omega_k(M)$ the space of $k$-dimensional currents on $M$, i.e., the topological  dual of  the space $\Omega^k(M)$  of smooth $k$-forms on $M$.  By definition, we have a pairing
\[
\lan-,-\ran: \Omega^k(M)\times \Omega_k(M)\to \bR, \;\;(\eta,C)\mapsto \lan\eta,C\ran.
\]
The orientation  of $M$  defines a   natural Poincar\'{e} duality map
 \[
 \Omega^{m-k}(M)\ni \omega\mapsto \omega^\dag\in \Omega_k(M),\;\;\lan\eta,\omega^\dag\ran :=\int_M \eta \wedge \omega,\;\;\forall \eta\in \Omega^{k}(M).
 \]
Given $\omega\in\Omega^{m-k}(M)$  we will refer to $\omega^\dag\in \Omega_k(M)$ as the \emph{current determined by the form} $\omega$.  By duality  we obtain  a boundary map
 \[
 \pa :\Omega_k(M)\to\Omega_{k-1}(M),\;\;\lan \eta,\pa C\ran:=\lan d\eta , C\ran,\;\;\forall C\in \Omega_{k}(M),\;\;\eta \in \Omega^{k-1}(M).
 \]
 A current $C$ is called closed if $\pa C=0$. 
 
 A generic section $\bu$ of $E$  is transversal to the zero section, $u\pitchfork 0$,  and its zero locus  is a smooth  submanifold   $Z_\bu\subset M$ of dimension $m-r$ equipped with a natural orientation.   The integration along this oriented submanifold  defines a  closed current $[Z_\bu]\in \Omega_{m-r}(M)$.
 
 The   Gauss-Bonnet-Chern theorem states  that,  for a generic section $\bu$,  the  $(m-r)$-dimensional closed currents $[Z_\bu]$ and the Poincar\'{e} dual  $\be(E,\nabla^E)^\dag$ are homologous, i.e.,
 \begin{equation}
\forall \bu\in C^\infty(E):\;\;\bu\pitchfork 0\Ra  [Z_\bu]-\be(E,\nabla^E)\in \pa \Omega_{m-r-1}(M).
 \label{GBC0}
 \end{equation}
 In view of DeRham's theorem \cite[\S 22, Thm. $17'$]{DR}, this is equivalent with the statement 
 \begin{equation}
\forall \bu\in C^\infty(E),\;\;\bu\pitchfork 0\Ra  \lan \eta , [Z_\bu]\ran =\int_M\eta \wedge \be(E,\nabla^E),\;\;\forall \eta\in\Omega^r(M),\;\;d\eta=0.
 \label{GBC1}
 \end{equation}
 
 \begin{remark} There exist more refined versions of   (\ref{GBC0})   which explicitly describe   locally integrable forms  $T=T(\bu,\nabla^E)$ such that we have the  equality of currents
 \[
  [Z_\bu]-\be(E,\nabla^E)=d T(\bu,\nabla^E).
  \]
  For details we refer to \cite{BZ, HL, MQ}. \qed
 \end{remark}
 
 \subsection{Overview of the paper} The first goal of this paper is to provide  a probabilistic proof and a refinement  of  (\ref{GBC1}).   Let us first  observe that if $\bu,\bv$ are two generic smooth sections of $E$, then the corresponding currents are homologous, i.e., 
 \[
 [Z_\bu]-[Z_\bv]\in\pa\Omega_{m-r-1}(M)\iff \lan\eta,[Z_\bu]\ran=\lan\eta,[Z_\bv]\ran,\;\;\forall \eta\in \Omega^{m-r}(M),\;\;d\eta=0 .
 \]
This shows that  if $\bu_1,\dotsc,\bu_n$ are generic sections of $E$ and $ p_1,\dotsc, p_n$ are positive weights  such that $p_1+\cdots +p_n=1$, then the average
 \[
 p_1[Z_{\bu_1}] +\cdots +p_n[Z_{\bu_n}]
 \]
 is a closed current homologous to each of the currents $[Z_{\bu_k}]$. 
 
 More generally, if $\bsP$ is a probability measure on $C^\infty(E)$ such that $\bsP$-almost surely a section $\bu$ intersects the zero section transversally, then the  expected  current 
 \[
 \bsE_{\bsP}([Z_\bu]):=\int [Z_\bu] \bsP(d\bu)
 \]
  is a current \emph{homologous to} the current defined by the zero locus of any  generic section $\bu_0$, i.e.,
  \begin{equation}  
  \int \lan \eta,[Z_\bu]\ran  \bsP(d\bu)=\lan\eta, [Z_{\bu_0}]\ran,\;\;\forall \eta\in \Omega^{m-r}(M),\;\;d\eta=0.
  \label{avcurr}
  \end{equation}
 An \emph{ensemble} of sections of $E$ is a pair $(\bsU,\bsP)$, where $\bsU\subset C^\infty(E)$ is a finite dimensional space and $\bsP$ is a probability measure on $\bsU$.  The first  main result of this paper   shows that  there exists a large supply of ensembles $(\bsU,\bsP)$  such that 
  
  \begin{itemize} 
  
  \item a section $\bu\in \bsU$ is $\bsP$-almost surely transversal to the zero section, and
  
  \item there exist  a metric $(-,-)_E$ and a connection $\nabla^E$, compatible  with $(-,-)_E$ such that  the expected current $\bsE_{\bsP}([Z_\bu])$ is \emph{equal to} the current determined by   the Euler form  $\be(E,\nabla^E)$, i.e.,  
    \[
 \bigl\lan\,\eta\, ,\, \bsE_{\bsP}([Z_\bu])\,\bigr\ran=\int_{\bsU}\lan \eta, [Z_\bu]\ran \bsP(d{\bu})=\int_M \eta \wedge \be(E,\nabla^E),\;\;\forall \eta\in \Omega^{m-r}(M).
 \]
 \end{itemize}
 
 We will  refer to an ensemble    $(\bsU,\bsP)$ with the above properties  as \emph{adapted to the metric $(-,-)_E$ and the connection  $\nabla^E$}. In the sequel, we will refer to a pair    consisting of a metric on a vector bundle and a connection  compatible with it as  a \emph{(metric,connection)-pair}. The first   step in our program is to produce  a large supply of examples of (metric,connection)-pairs   for which we can \emph{explicitly} construct  adapted  ensembles $(\bsU,\bsP)$.    

Fix a  finite dimensional real  vector space $\bsU$ equipped with  a Euclidean   inner product $(-,-)_\bsU$.   We form the  trivial real vector bundle
\[
\underline{\bsU}_M:= \bsU\times M\to M.
\]
Assume that $E\to M$ is   an oriented subbundle  of rank $r$ of $\underline{\bsU}_M$. The metric $(-,-)_\bsU$ on $\bsU$ induces   a metric $(-,-)_E$ on $E$.  For each $\bx\in M$ we denote by $P_\bx$ the orthogonal projection   $\bsU\to E_\bx$.  The trivial connection $d$ on $\underline{\bsU}_M$ induces a connection $\nabla^E=Pd$ on $E$.   We will  call \emph{special} a (metric, connection)-pair $\bigl(\; (-,-)_E,\nabla^E\,\bigr)$ constructed  as above,  via an embedding of $E$ in a trivial  vector bundle equipped with a trivial metric and the trivial connection.

Any $\bu\in\bsU$ defines a section $S_\bu^E$ of $E$ given by
\[
S_\bu^E(\bx)=P_\bx u,\;\;\forall \bx\in M.
\]
We thus get a linear map $S^E:\bsU\to C^\infty(E)$, $\bu\mapsto S_\bu^E$, whose range is the finite dimensional space 
\[
\widehat{\bsU}:=\bigl\{ \, S_\bu^E;\;\;\bu\in\bsU\,\bigr\}\subset C^\infty(E).
\]
The metric  on $\bsU$ induces a Gaussian probability measure  on ${\bsU}$; see  (\ref{met_gauss}). Its pushforward by $S^E$ is  a Gaussian probability  measure  $\gamma_{\bsU}$ on $\widehat{\bsU}\subset C^\infty(E)$.

   Theorem \ref{th: main_int}(i) shows that,    $\gamma_{\bsU}$-almost surely,   a  section   $\hat{\bu}\in\widehat{\bsU}$ intersects transversally the zero section of $E$. We denote by $[Z_{\hat{\bu}}]$  the current of integration  defined by zero locus of $\hat{\bu}$.   

The key integral formula (\ref{intfor}) in  Theorem \ref{th: main_int} shows that the expectation of the random current $[Z_{\hat{\bu}}]$ is equal  to   the current determined by $\be(E,\nabla^E)$, i.e., 
\begin{equation}
\bigl\lan\,\eta\, ,\, \bsE_{\gamma_\bsU}([Z_{\hat{\bu}}])\,\bigr\ran=\int_{\widehat{\bsU}}\lan \eta, [Z_{\hat{\bu}}]\ran \gamma_\bsU(d\hat{\bu})=\int_M \eta \wedge \be(E,\nabla^E),\;\;\forall \eta\in \Omega^{m-r}(M).
\label{BGC2}
\end{equation}
In other words, the ensemble $(\widehat{\bsU},\bgamma_\bsU)$ is adapted to the special pair $(\,(-,-)_E,\nabla^E)$. 

\begin{remark} As explained in \cite{BZ, HL, MQ}, there are many natural ways to \emph{explicitly} associate to each section $\bu\in C^\infty(E)$,  and any connection $\nabla$ on $E$ compatible with $(-,-)_E$,   a locally integrable form  $T(\bu,\nabla)$  of degree $(2h-1)$ on $M$ satisfying  the equality of currents 
\begin{equation}
[Z_\bu]-\be(E, \nabla)=d T(\bu,\nabla).
\label{PL}
\end{equation}
Such an equality generalizes   the Poincar\'{e}-Lelong formula in  complex analysis and it clearly implies (\ref{GBC0}).

 Let $\widehat{\bsU}$, $\bgamma_\bsU$, $(-,-)_E$ and $\nabla^E$ be as  in (\ref{BGC2}). Averaging (\ref{PL})  over $\bu\in\widehat{\bsU}$ with respect to the measure $\bgamma_\bsU$ we deduce from (\ref{BGC2}) that
\[
\int_{\widehat{\bsU}}d T(\bu,\nabla)\bgamma_\bsU(d\bu)=\int_{\widehat{\bsU}}\bigl(\, [Z_\bu]-\be(E,\nabla)\,\bigr)\bgamma_\bsU(d\bu)= \be(E,\nabla^E)-\be(E,\nabla).
\]
In particular, when $\nabla=\nabla^E$ we have
\begin{equation}
\int_{\widehat{\bsU}}d T(\bu,\nabla^E)\bgamma_\bsU(d\bu)=0.
\label{HL1}
\end{equation}
Conversely, the equality (\ref{HL1}) implies (\ref{BGC2}). One could then be tempted to prove  (\ref{BGC2}) by   proving  a stronger version of  (\ref{HL1}), namely
\begin{equation}
 \int_{\widehat{\bsU}} T(\bu,\nabla^E)\bgamma_\bsU(d\bu)=0.
 \label{PL1}
 \end{equation}
  The complexity of the  description of the transgression form  $T(\bu,\nabla^E)$ has discouraged us from attempting to verify the validity of  (\ref{PL1}). We have instead opted on a different approach based on the double-fibration trick frequently used in  integral geometry. \qed
\end{remark}

Obviously, the  equality (\ref{BGC2}) implies  (\ref{GBC1}) for  special (metric, connection)-pairs on $E$. Since the Euler form is   gauge invariant, we see that (\ref{BGC2}) is valid  if we replace the special connection $\nabla^E$ with a connection that is gauge equivalent to it. Here the gauge group is the group of orientation preserving, metric preserving automorphisms of $E$.    On the other hand, we have the following result.

\begin{proposition}   Any (metric, connection)-pair $(\bsi,\nabla)$ on an oriented vector bundle  $E\to M$  is gauge equivalent to a  special pair.
\end{proposition}

\begin{proof}    The proof is carried out in two steps.

\medskip   

\noindent {\bf 1.}  The pullback of a  special (metric, connection)-pair is a special  (metric connection)-pair.  Suppose that $(\bsi,\nabla)$ is a special (metric, connection)-pair  on the subbundle $E\to M$ of the trivial bundle $\underline{\bsU}_M$.

If $X$ is a smooth manifold    and $\Phi: X\to M$ is a smooth map, then we get a vector  bundle  $\Phi^*E$ over $X$ equipped with the  metric $\Phi^*\bsi$ and the compatible connection  $\Phi^*\nabla$.   The bundle $\Phi^*E$ is a subbundle of the trivial vector bundle
\[
\Phi^*\underline{\bsU}_M=\underline{\bsU}_X
\]
equipped with the trivial metric. Then $\Phi^*\bsi$ is the induced metric on $\Phi^* E$ as a subbundle of the metric bundle $\underline{\bsU}_X$ and $\Phi^*\nabla$ is the connection induced  via orthogonal projection from the  trivial connection on $\underline{\bsU}_X$.

\smallskip

\noindent {\bf 2.} Consider the Grassmannian  $\Gr^+_r(\bsU)$ of $r$-dimensional oriented subspaces of $\bsU$.    Denote by $\eT_r(\bsU)\to \Gr_r^+(\bsU)$ the associated tautologial oriented vector bundle. A metric $h$ on $\bsU$ induces  a  metric $\bsi_h$,  and a  compatible connection $\nabla^h$ on $\eT_r(\bsU)$. The pair $(\si_h,\nabla^h)$ is special.

  In \cite[Thm. 1, 2]{NR1}  Narasimhan and Ramanan have shown that  for any smooth,      real oriented vector bundle  $E\to M$ and any  (metric, connection)-pair $(\bsi,\nabla)$ on $M$ there exists a finite dimensional Euclidean space $(\bsU,h)$ and a smooth map $\Phi: M\to\Gr^+_r(\bsU)$ such that
 \begin{equation}\label{NRm}
 E=\Phi^*\eT_t(\bsU),\;\;\bsi=\Phi^*\bsi_h
 \end{equation}
 and the connection $\nabla$ is gauge equivalent to $\Phi^*\nabla^h$.  We will refer to such maps as  \emph{Narasimhan-Ramanan maps}.
\end{proof}

Putting together all of the above we obtain the first main result of this paper.

\begin{theorem} Suppose that $E\to M$ is a  smooth real oriented vector bundle of rank $r=2h$ over a smooth compact oriented manifold $M$ of dimension $m$. For any  metric $\bsi$ on $E$ and any connection $\nabla$ on $E$ compatible with $\bsi$ there exists   a finite dimensional  subspace $\widehat{\bsU}\subset C^\infty(E)$ and a Gaussian measure $\gamma$ on $\widehat{\bsU}$ such that,    $\gamma$-almost surely, a section $\hat{\bu}\in \widehat{\bsU}$ is  transversal to  the  zero section  and the expectation of the random zero-locus-cycle
\[
\widehat{\bsU}\ni\hat{\bu}\mapsto[Z_{\hat{\bu}}]\in \Omega_{m-r}(M)
\]
is equal to the current  determined by  the Euler form of $\nabla$.\label{th: A} \qed
\end{theorem}

Clearly the above result implies the classical   Gauss-Bonnet-Chern theorem,  but it has a glaring {\ae}sthetic flaw since it gives no idea on the nature of the ensemble $(\widehat{\bsU},\bgamma_\bsU)$. Its relationship to the geometry of $(E,\si, \nabla)$   is hidden in the  details of the proofs of \cite[Thm. 1,2]{NR1}. Those proofs   show that, to produce   such an ensemble, we need to make  several    noncanonical choices: a choice of a  gluing cocycle for $E$  and    a choice of a collection of locally  defined  $\uso(n)$-valued $1$-forms describing   $\nabla$. The dependence of $(\widehat{\bsU}, \bgamma_\bsU)$ on these choices  is  nebulous.

The second goal of the paper  is to  address this issue. To formulate   our  second main result  we need to   describe an alternate way of producing special (metric, connection)-pairs.

Suppose that  $\bsU\to C^\infty(E)$ is a    finite dimensional space  of  sections of $E$ large enough so that it satisfies the \emph{ampleness condition}
\begin{equation}
\spa\bigl\{\bu(\bx);\;\;\bu\in\bsU\,\bigr\}=E_\bx,\;\;\forall \bx\in M.
\label{ample0}
\end{equation}
In particular, for every $\bx\in M$,  the evaluation map
\[
\ev_\bx:\bsU\to E_\bx,\;\;\bu\mapsto\ev_\bx \bu:=\bu(\bx)
\]
is onto,  so that its dual $\ev_\bx^*: E_\bx^*\to \bsU^*$ is one-to-one. Thus,   the dual bundle $E^*$ is naturally a subbundle of $\underline{\bsU}_M^*$.  

If we fix an inner product $(-,-)_\bsU$ on $\bsU$, then we can identify   $\bsU$ with $\bsU^*$ and we  can view $E$ as a subbundle of the trivial bundle $\underline{\bsU}_M$. Fixing a Euclidean metric on $\bsU$   is equivalent with fixing a nondegenerate Gaussian probability   measure $\gamma_\bsU$ on $\bsU$; see Subsection \ref{ss: 41}.   This  discussion shows that to any  nondegenerate   Gaussian probability measure  on an ample subspace $\bsU\subset C^\infty(E)$ we can cannonically  associate a special (metric, connection)-pair on $E$.   

\begin{definition}\label{def: smaple} A \emph{sample subspace} of $C^\infty(E)$  is  a pair $(\bsU,\gamma)$,  where $\bsU\subset C^\infty(E)$ is  an ample finite dimensional subspace and $\gamma$ is a nondegenerate Gaussian measure on $\bsU$.   The space $\bsU$ is called the \emph{support} of the sample space. \qed
\end{definition} 

 Thus, to any sample subspace $(\bsU,\gamma)$  of $C^\infty(E)$ we can associate a  special (metric,  connection)-pair on $E$.  Theorem \ref{th: main_int} shows that the   expectation  of the random current defined  by the zero-locus  of a random $\bu\in\bsU$   is equal to the current determined by  the Euler form of the associated special (metric, connection)-pair.   

In Theorem \ref{th: scl}  we show that any (metric, connection)-pair  $(\bsi_0,\nabla^0)$    on $E$  can be approximated in a rather \emph{explicit} fashion by  special (metric,connection)-pairs    associated to  sample subspaces  canonically and explicitly determined by $(\bsi_0,\nabla^0)$.

  More precisely, in Theorem \ref{th: scl} we produce \emph{explicitly}    a family of  sample spaces $(\bsU_\ve,\gamma_\ve)_{\ve>0}$ with associated special (metric,connection)-pairs  $(\bsi_\ve,\nabla^\ve$)  satisfying the following properties.
 \begin{subequations}
 \begin{equation}
 \ve_1<\ve_2\Ra \bsU_{\ve_1}\supset \bsU_{\ve_2},
 \label{scla}
 \end{equation}
 \begin{equation}
 \bigcup_{\ve>0}\bsU_\ve \;\;\mbox{is dense in $C^\infty(E)$},
 \label{sclb}
 \end{equation}
 \begin{equation}
 \Vert|\bsi_\ve-\bsi_0\Vert_{C^0}=o(1),\;\;\mbox{as $\ve\to 0$}
 \label{sclm}
 \end{equation}
  \begin{equation}
 \Vert \nabla^\ve-\nabla^0\Vert_{L^{1,p}} +\Vert F^\ve-F^0\Vert_{C^0}= o(1)\;\;\mbox{as $\ve\to 0$},\;\;\forall p\in (1,\infty)
\label{sclc}
\end{equation}
\end{subequations}
where $L^{1,p}$ denotes the  Sobolev space of   distributions   with first order derivatives in $L^p$ while   $F^\ve$ denotes the curvature of $\nabla^\ve$.

For each $\ve$, the sample space $\bsU_\ve$ produces  a smooth map $\Psi_\ve: M\to\Gr^+_r(\bsU_\ve)$.  If $\nah^\ve$ denotes the canonical connection of the tautological vector bundle over $\Gr_r^+(\bsU_\ve)$, then 
\[
\Psi_\ve^*\nah^\ve=\nabla^\ve.
\]
Theorem \ref{th: scl}  shows that $\Psi_\ve^*\nah^\ve$ is very close to $\nabla^0$  for $\ve$ small. From this  perspective we can  view  Theorem \ref{th: scl} as  providing a probabilitic construction of   approximate  Narasimhan-Ramanan maps; see (\ref{NRm}).

Let us observe that Theorem \ref{th: scl} also  implies   the Gauss-Bonnet-Chern theorem for  the  pair $(\bsi_0,\nabla^0)$, but without appealing to the results of Narasimhan and Ramanan \cite{NR1}. Indeed,  (\ref{BGC2}) implies  that for any $\ve>0$ and any   $\eta\in \Omega^{n-r}(M)$ we have
\[
  \int_{\bsU_\ve}\lan \eta,[Z_\bu]\ran \gamma_\ve(d\bu)=\int_M\eta\wedge \be(E,\nabla^\ve)\iff \bsE\bigl([Z_\bu]|\bu\in\bsU_\ve\bigr) =\be(E,\nabla^\ve)^\dag.
\]
We let $\ve\to 0$  and we  conclude   from (\ref{sclc}) that, 
\begin{equation}
 \lim_{\ve\to 0}\int_{\bsU_\ve}\lan \eta,[Z_\bu]\ran \gamma_\ve(d\bu)=\int_M\eta\wedge \be(E,\nabla^0),\;\;\forall \eta\in \Omega^{m-r}(M).
\label{main}
\end{equation}
On the other hand, (\ref{avcurr}) shows  that for any  generic section $\bu_0$ of $E$, any \emph{closed} form $\eta\in\Omega^{m-r}(M)$ and any $\ve>0$ we have
\[
\lan \eta, [Z_{\bu_0}]\ran =  \int_{\bsU_\ve}\lan \eta,[Z_\bu]\ran \gamma_\ve(d\bu).
\]

As we have mentioned earlier, the spaces $\bsU_\ve$  can be constructed \emph{explicitly}.   We were led to these spaces guided by probabilistic ideas, but they can be given a purely analytic description.  In either interpretation, these spaces  depend on two additional choices.    

The first choice is \emph{a Riemann metric $g$ on $M$}. Form the covariant Laplacian  $\Delta_0=(\nabla^0)^*\nabla^0: C^\infty(E)\to C^\infty(E)$.  It has a discrete spectrum
\[
\spec(\Delta_0)=\lambda_1\leq \lambda_2\leq \cdots.
\]
Let  $(\Psi_n)_{n\geq 1}$ be a complete orthonormal   family of $L^2(E)$ consisting of eigensections of $\Delta_0$,
\[
\Delta_0\Psi_n=\lambda_n\Psi_n.
\]
Our first candidate for  the approximating family  $\bsU_\ve$ is defined by
\[
\bsU_\ve:=\spa\Bigl\{\,\Psi_n;\;\;\lambda_n\leq\ve^{-2}\,\Bigr\}.
\]
As metric $\bsi_\ve$ on $\bsU_\ve$ we  use the $L^2(E)$-inner product rescaled by  the factor $\ve^m$.    The family $(\bsU_\ve)_{\ve>0}$ satisfied  (\ref{scla}) and (\ref{sclb}) and with a little work  it can be shown that is also satisfies (\ref{sclm}).   However,  proving that this family of sample spaces  also satisfies (\ref{sclc}) is fraught with many technical difficulties.    To avoid  them  we need to tweak this approach. 

Let  $\bchi:\bR\to [0,\infty)$  be  the characteristic function of the interval $[-1,1]$. Observe that  $\bsU_\ve$  can alternatively be defined as the range of the  smoothing operator $\bchi(\ve\sqrt{\Delta_0})$.  We now make  our second choice and  we  \emph{fix a compactly supported, smooth, even function} $w:\bR\to[0,\infty)$ such that $w(0)>0$.  Intuitively, we think of $w$ as a smooth approximation for $\bchi$. For any $\ve>0$ we   have a smoothing operator
\[
W_\ve := w\bigl(\ve\sqrt{\Delta_0}\,\bigr): L^2(E)\to L^2(E).
\]
The operator $W_\ve$  is  symmetric, nonnegative definite and has finite dimensional range $\bsU_\ve:= {\rm Range}\,W_\ve$.     Clearly the family $(\bsU_\ve)_{\ve>0}$ satisfies (\ref{scla}) and (\ref{sclb}). In particular, this shows that $\bsU_\ve$ is ample if $\ve$ is sufficiently small. 

The space  $\bsU_\ve$ is also a  $W_\ve$-invariant subspace of $L^2(E)$ and the restriction  of $W_\ve$ to $\bsU_\ve$ is invertible because $w(0)\neq 0$.  The Gaussian measure  $\gamma_\ve$ is then defined by
\[
\gamma_\ve(d\bu) =\frac{1}{\sqrt{\det 2\pi W_\ve}}  e^{-\frac{1}{2}(W_\ve^{-1}\bu,\bu)_0} |d\bu|_0,
\]
where $(-,-)_0$ denotes the $L^2$-inner product on $\bsU_\ve$ and $|d\bu|_0$ denotes the associated Lebesgue measure. In Theorem \ref{th: scl} we prove that the family of sample spaces $(\bsU_\ve,\gamma_\ve)$  defined in this fashion satisfy  all the properties (\ref{scla})-(\ref{sclc}).

The  sample space $(\bsU_\ve,\gamma_\ve)$ as defined above has a simple  probabilistic interpretation.       A random  section $\bu_\ve\in\bsU_\ve$ is   a random linear superposition
\begin{equation}
\bu_\ve=\sum_n X_n^\ve \Psi_n,
\label{rand_super}
\end{equation}
where the coefficients  $X_n^\ve$ are  independent   normal random variables with mean $0$ and variances
\[
\var(X_n^\ve)=w(\ve\sqrt{\lambda_n}).
\]
The correlation kernel of the random section $\bu_\ve$  coincides with the Schwartz kernel of $W_\ve$, and the  connection of $E$ determined by the Gaussian  ensemble $(\bsU_\ve,\bgamma_\ve)$  is  a special case of the $L$-$W$ connection  in  \cite[Prop. 1.1.1]{ELL}. We provide  probabilistic descriptions of this connection and its curvature  in  Subsection \ref{ss: 42}. These descriptions play a key role in the proof of Theorem \ref{th: scl}.

 Note that for  any given $\ve$ we have $w(\ve\sqrt{\lambda_n})=0$ if $n$ is sufficiently large so that the sum (\ref{rand_super}) consists of finitely many terms. If $w=1$ in a neighborhood of $0$, then as $\ve\to 0$  the above     random   linear superposition   formally converges  to a   random series
\[
\sum_n X_n^0 \Psi_n,
\]
where the coefficients  $X_n^0$ are  independent standard  normal random variables. This is very similar to the   classical scalar white noise.  In fact,  as explained in \cite{GeVi2},  the above series converges in the sense of distributions to   a generalized Gaussian random process  called  white noise. For this reason we will refer to the $\ve\to 0$ limit as the  white-noise limit. Thus, the differential geometry of $(E,\bsi_0,\nabla^0)$  is determined  by the white-noise approximation regime    defined by the family of random sections $\bu_\ve$, $\ve >0$. Observe also that the equality  (\ref{main})  has the following nice consequence.
\begin{corollary}\label{cor: sz}
\[
\lim_{\ve\searrow 0}\bsE
\bigl(\,[Z_{\bu_\ve}]\,\bigr)=\be\bigl(\, E,\nabla^E\,\bigr)^\dag.\proofend
\]
\end{corollary}

\subsection{Related work}        The results  in this paper   take place on real manifolds and real vector bundles and deal with two themes: the distribution of zeros of random  sections and     the connections between the statistics of such suctions and the geometry of the bundle.

This line of investigation originates in the groundbreaking work of M. Kac \cite{Kac43} and S.O. Rice \cite{Rice45} who studied the distribution of zeros of certain  random functions of one real  variable.  One outcome of their work is the celebrated  Kac-Rice formula that gives an explicit   description of the expected distribution of  zeros  of such functions; see \cite{EK}.  

As explained  e.g. in \cite{EK, ShVa}, the Kac-Rice  formula  has an  extension to  complex valued random functions of one complex  variable that can be used to study  the distribution of complex zeros of various ensembles  of random complex polynomials.  One such ensemble   is obtained by regarding the degree $d$ polynomials  in one complex variable  as holomorphic  sections of the  degree $d$ holomorphic line bundle over $\mathbb{CP}^1$.  If  the ensemble of sections is   unitarily equivariant, then the  expected  distribution of zeros  approaches the uniform  (invariant) distribution on $\mathbb{CP}^1$ as $d\to\infty$; see \cite[\S 8]{EK}.  The same type of equidistribution  phenomenon was observed in the pioneering work of S. Nonnenmaker and A. Voros \cite{NV} where, among many other things,   the authors studied   the distributions  of zeros  of random holomorphic sections of holomorphic line bundles over elliptic  curves. 

In  the holomorphic context, the inherent rigidity allows for more precise conclusions in arbitrary dimensions.  Chapter 5 of the  monograph \cite{MaMa} by X. Ma and G. Marinescu   contains a very nice exposition of these  developments. We mention below a few of them.
 
  In  \cite{SZ1}, Schiffman and Zelditch  have  investigated  random holomorphic  sections  of  $L^n$, $n\gg 1$,  where $L$ is an ample hermitian holomorphic  line bundle $L$ over a compact K\"{a}hler    manifold $M$.      Our Corollary \ref{cor: sz}   has the same flavor as  \cite[Thm. 1.1]{SZ1} ; see also \cite[Thm.5.3.3.]{MaMa}.

The large  $n$ limit is  conceptually  similar to the white noise limit we employ in this paper  although   the technical details are quite different.    G. Tian \cite{Tian} and W.-D. Ruan \cite{Ruan} have  shown  how   to use the ensemble of holomorphic sections of $L^n$, $n\gg 1$, to produce $C^\infty$-approximations of the curvature of $L$.     D. Catlin \cite{Cat} and  S. Zelditch \cite{Z}    gave alternate proofs of this fact where the probabilistic  features  are  easier to glean.  Our proof of Theorem \ref{th: scl} is similar in spirit to theirs.  

In the last few years there has been a flurry of work, e.g.,   \cite{CM, CMM, CMN, DMS},   concerning  the statistics of the zero sets  of  random  holomorphic sections of $L^n$ in the case when $M$ is noncompact/singular.

 There are  fewer similar results  for real manifolds, and they are typically harder to come by.  This should come as no surprise since even the simplest problem, that of counting the number of real roots of a polynomial, is tricky.  We want to mention a few noteworthy contributions.  
 
  The first is  the work of R. Adler and J. Taylor \cite{AT} on the curvature measures of the zero sets of  random maps  form a manifold a real vector space. The second  is  the (ongoing) investigation of  F. Nazarov and M. Sodin \cite{NS1,NS2}  concerning the topology  of the zero set of a  random function on a manifold.  Finally, we want to mention  the recent work of A. Lerario and E. Lundberg \cite{LL} on a   probabilistic version of  Hilbert's 16th problem.
 
 The original   one-dimensional  Kac-Rice formula admits  wide ranging  higher-dimensional generalizations,  \cite{AT,AzWs}.    In \cite{N5} the first author used such a general Kac-Rice formula  to extend Theorem \ref{th: main_int} to arbitrary Gaussian ensembles of random sections, that is,  arbitrary Gaussian measures on $C^\infty(E)$, not necessarily supported on finite dimensional sample spaces.

  In Theorem \ref{th: scl}  we produce only $C^0$-approximations of the curvature  of the vector bundle.  However, in the special case when $E=TM$, $\bsi_0$ is a Riemannian metric on $M$   and $\nabla^0$ is the associated Levi-Civita  connection, then the results in \cite{AC, Pet}  imply that (\ref{sclm}) can be refined to a $C^\infty$-convergence of $\bsi_\ve$ to the Riemann metric $\bsi_0$.

In  \cite{N4} the first author  has investigated critical sets of   random functions on a compact Riemann manifold.   The critical points of  functions are zeros of rather special  sections of the cotangent bundles, namely zeros of exact $1$-forms.  In \cite[Thm.1.7]{N4}   it was shown  that  the  geometry of a Riemann manifold   is determined by the statistics of the differentials of  random functions on it.  This is similar in flavor  with   Theorem \ref{th: scl} in the present paper. However  \cite[Thm.1.7]{N4} does not follow from the apparently more general  Theorem \ref{th: scl} in this paper.

 \subsection{Organization of the paper}  The main body of the paper consists of two sections. In Section \ref{s: 2} we prove  our main integral formula  Theorem \ref{intfor} which states that if $(\bsU,\gamma)$ is a sample space of $C^\infty(E)$, then the  expectation of the zero-locus-current of a random   section $\bu\in\bsU$  is  equal to the current determined by the Euler form of the  special connection on $E$ induced  by this sample space.   The proof  relies on the ubiquitous double-fibration trick. We evaluate the various  intervening integrals   using the  theory  of orthogonal invariants like in  Weyl's proof of his tube formula \cite{Wetube}.  
 
 Section \ref{s: 3}   contains the proof of our   reconstruction result, Theorem \ref{th: scl}.  It boils down to a   detailed understanding of the Schwartz kernel  of  the smoothing operator $w(\ve\sqrt{\Delta_0})$.
 
 We approach     this problem using  the wave kernel  technique pioneered by L. H\"{o}rmander \cite{Hspec}.  The fact that our operators are not  scalar  makes the identification  of  various terms in the asymptotic expansion of this kernel more challenging. We achieve this by gradually reducing the computation of these terms  to the special case  involving the heat kernel. The estimate (\ref{sclc}) is trickier and  follows using a method reminiscent  to the one employed by K. Uhlenbeck in \cite{U}.
 
 \subsection*{Acknowledgments.}  We want to thank the anonymous  referee for the  helpful comments and critique.

\section{A finite dimensional integral  formula} 
\label{s: 2}
\setcounter{equation}{0} 

\subsection{The setup} Suppose that $M$ is a compact oriented smooth manifold of dimension $m$ and $E\to \bR$ is a real, oriented vector bundle of  even rank $r=2h$.  We fix a finite dimensional  space $\bsU\subset  C^\infty(E)$, 
\[
\dim\bsU=N.
\]
  Any $\bx\in M$ defines a linear evaluation map
\[
\ev_\bx:\bsU\to E_\bx,\;\;\bsU\ni\bu \mapsto \bu(\bx).
\]
We assume that $\bsU$  satisfies the ampleness condition (\ref{ample0}). The dual map $\ev_\bx^*: E^*_\bx\to \bsU^*$ is an injection  and the   family  $(\ev_\bx^*)_{\bx\in M}$  describes an inclusion  of $E^*$ as a subbundle of the trivial vector bundle $\underline{\bsU}^*_M$.

We fix an Euclidean metric $(-,-)_\bsU$ on $\bsU$. It induces a metric $(-,-)_{\bsU^*}$  on $\bsU^*$. The inclusion  
\[
\ev^*:E^*\to \underline{\bsU}_M^*
\]
 induces a metric $(-,-)_{E^*}$ on  the bundle $E^*$ and, by duality, a metric $(-,-)_E$  on $E$.   

The evaluation map   $\ev_\bx:\bsU\to E_\bx$   can be identified with the  orthogonal projection. To emphasize this aspect, we will use the alternate notation $P=P_{\bx}:=\ev_\bx$. We also set 
$Q=Q_{\bx}=\one-P_{\bx}$.

If we choose   an orthonormal basis $(\Psi_k)_{1\leq k\leq N}$ of $\bsU$,  then  we can describe the projection $P_\bx$ in  the concrete form
 \[
 P_\bx \bu =\sum_{k=1}^N (\bu,\Psi_k)_\bsU\Psi_k(\bx).
 \]
Let us point a confusing fact. A \emph{fixed}  vector $\bu\in\bsU$ can   be viewed as a constant section of the trivial  bundle $\underline{\bsU}_M$ and also, by definition, as a section of  $E$. As such it is given by the  smooth map
\[
S^E_\bu: M\to \bsU,\;\; S^E_\bu(\bx)= \ev_\bx\bu=P_{\bx}\bu.
\]
We denote by $K$  the subbundle of $\underline{\bsU}_M$ defined by the kernels of the above projections, $K:=\ker  P$. Note that
\[
E=K^\perp,\;\; E\oplus K\cong \underline{\bsU}_M=\bsU\times M.
\]
If we denote by $d$ the trivial connection on $\underline{\bsU}_M$, then  we obtain a connection on $\nabla^E$ on  $E$ compatible with the metric $(-,-)_E$,
\[
\nabla^E:=P dP.
\]
We denote by $F^E$ the curvature of the connection $\nabla^E$ on $E$ and by $\be(E,\nabla^E)$ the associated Euler form defined as in (\ref{euler})
\[
\be(E,\nabla^E)=\frac{1}{ (2\pi)^{h} } \pf(-F^E)\in\Omega^r(M),\;\;r=2h.
\]
If a  section $\bu\in\bsU$ is transversal to the zero section, $\bu\pitchfork 0$, then its zero set
 \[
 Z_\bu:=\bigl\{ \, \bx \in M;\;\;\bu(x)=0\,\bigr\}
 \]
 is  a     compact submanifold of $M$ of codimension $r$. We denote by $T_{Z_\bu}M$ its normal bundle in $M$,
 \[
 T_{Z_\bu}M:=TM|_{Z_\bu}/TZ_\bu.
 \]
 Given  any connection  $\nabla$ on $E$ we obtain a linear map 
 \[
 \nabla_{\bullet} \,\bu: (TM)|_{Z_\bu}\to E|_{Z_\bu}
 \]
 which vanishes  along $TZ_\bu$ and thus induces a bundle morphism
 \[
 \mathfrak{a}_\bu:  T_{Z_\bu}M\to E|_{Z_\bu}
 \] 
 that is independent of the choice of $\nabla$. We will refer to $\mathfrak{a}_\bu$ as the \emph{ajunction morphism}.

 The transversality $\bu\pitchfork 0$   is equivalent to the fact that  $\mathfrak{a}_\bu$ is a bundle isomorphism. The orientation on $E$ induces via the adjunction morphism  an orientation in the normal bundle $(TM)|_{Z_\bu}$ and thus an orientation on  $Z_\bu$ uniquely  determined by the requirement
 \[
 {\rm orientation}\, TM|_{Z_\bu}={\rm orientation}\, (Z_\bu)\wedge {\rm orientation}\, (T_{Z_\bu} M).
 \]
 Let us point out that since $Z_\bu$ has \emph{even} codimension we have
 \[
 {\rm orientation}\, (Z_\bu)\wedge {\rm orientation}\, (T_{Z_\bu} M)= {\rm orientation}\, (T_{Z_\bu} M)\wedge {\rm orientation}\, (Z_\bu).
 \]
 We denote by $[Z_\bu]\in \Omega_{m-r}(M)$ the integration current  defind by the submanifold  $Z_\bu$  equipped with the above orientation.

 \begin{theorem}     Let $E\ra M$ be a   real oriented, smooth vector bundle of rank $r=2h$ over   the compact oriented   smooth manifold $M$. Fix a   subspace  $\bsU\subset C^\infty(E)$   of  dimension $\dim \bsU=N<\infty$ satisfying  the ampleness condition (\ref{ample0}).  Fix an Euclidean inner product $(-,-)_\bsU$ on $\bsU$ and denote by $\gamma_\bsU$ the    Gaussian measure on $\bsU$ determined by this inner product,
 \[
 \gamma_\bsU(d\bu):=\frac{1}{(2\pi)^{\frac{N}{2}}} e^{-\frac{|\bu|^2}{2}} d\bu.
 \]
 Then the  following   hold.
 
 \begin{enumerate} 
 
 \item A section $\bu\in\bsU$  almost surely  intersects  transversally the zero section   of $E$ and thus we obtain a random    current
 \[
 \bsU\ni \bu \mapsto [Z_\bu]\in \Omega_{m-r}(M).
 \]
 \item   The expectation of    this  random current is the current determined by  the Euler form $\be(E,\nabla^E)$
 \[
 \bsE_{\gamma_\bsU}( [Z_\bu])=\be(E,\nabla^E)^\dag.
 \]
 More precisely,
 \begin{equation}
\int_\bsU\lan \eta, [Z_\bu]\ran d\gamma_\bsU(du)=\frac{1}{(2\pi)^{\frac{r}{2}}}\int_M \eta \wedge \pf(-F^E),\;\;\forall \eta\in \Omega^{m-r}(M).
 \label{intfor}
 \end{equation}

 \end{enumerate}
 \label{th: main_int}
 \end{theorem}

 The proof of the  the integral formula (\ref{intfor})  is  based  on   Gelfand's double fibration trick,   \cite{APF, GGSh}. Its formulation  relies  on  two versions of the coarea formula.    We describe these versions below.
 
 \subsection{The coarea formula}  Suppose that $X,Y$ are \emph{oriented} smooth manifolds of dimensions
 \[
 \dim X=N\geq n=\dim Y.
 \]
 Asume further that  that we are given a smooth map $\pi: X\to Y$. For any  regular value $y\in Y$ of $\pi$ the fiber $X_y:=\pi^{-1}(y)$ is a smooth submanifold of $X$ of codimension $n$ and its conormal bundle $T^*_{X_y}X$  is naturally isomorphic with $\pi^*T^*Y|_{X_y}$ and thus   it has a natural orientation. We orient  $X_y$   using the \emph{fiber-first convention}, i.e.,
 \[
 {\rm orientation}\,(X)={\rm orientation}\,(X_y)\wedge {\rm orientation}\, T^*_{X_y}X.
 \]
  Suppose that $\omega_Y\in \Omega^{n}(Y)$ is a volume form on $Y$, i.e., a nowhere vanishing  top-degree form on $Y$. Fix a smooth function $\rho_Y: Y\to\bR$ and a form $\eta\in \Omega^{N-n}(X)$  such that
  \[
-\infty<  \int_{X_y}\eta<\infty
\]
for any regular value $y$ of $\pi$.     Sard's theorem  implies that   $y$ is a regular  value of $\pi$ for almost all $y\in Y$.   

The first version of the coarea formula  states that the function
\[
Y\ni y\mapsto  \int_{X_y}\eta\in\bR
\]
is  Lebesgue measurable and 
\begin{equation}
\int_Y\left(\int_{X_y}\eta\right)\rho_Y(y)\omega_Y=\int_X \eta \wedge \pi^*(\rho_Y\omega_Y),\;\;\eta\in \Omega^c(X).
\label{caf1}
\end{equation}

For the second version of  the  coarea formula we choose a top degree form $\alpha\in\Omega^{N}(X)$.  If $y_0\in Y$ is  a regular value of $\pi$, then there is an induced  \emph{Gelfand-Leray residue form}
\[
\frac{\alpha}{\pi^*\omega_Y}\in \Omega^{N-n}(X_{y_0}).
\]
It is locally constructed as follows. Fix a point $p_0\in X_y$ and local coordinates  $(x^1, \dotsc,x^N)$ on $X$ in a neighborhood $U$ of  $p_0$  and coordinates $(y^1,\dotsc, y^n)$ on $Y$ in a neighborhood $V$ of  $y_0=\pi(p_0)$ such that, in these coordinates,  the smooth map $\pi$  is linear  and  described by the functions
\[
y^i(x)=x^{N-n+i},\;\;\forall i=1,\dotsc, n.
\]
In the coordinates  $(y^i)$  the volume form $\omega_Y$    has the form
\[
\omega_Y=a(y) dy^1\wedge \cdots \wedge dy^n,
\]
where $a \in C^\infty(V)$ is a nowhere vanishing function. Now choose a form $\beta\in\Omega^{N-n}(U)$ such that
\[
\beta\wedge  a\bigl(\,x^{N-n+1},\dotsc, x^N \,\bigr)dx^{N-n+1}\wedge \cdots \wedge dx^N= \alpha.
\]
Ther restriction of $\beta$ to $X_{y_0}\cap Y$ is an $(N-n)$-form  on $X_{y_0}\cap U$  that is independent of all the choices and it is the Gelfand-Leray residue $\frac{\alpha}{\pi^*\omega_Y}$.  

The second version of the coarea formula that we will need   takes the form
\begin{equation}
\int_X \alpha =\int_Y\left( \int_{X_y}\frac{\alpha}{\pi^*\omega_Y} \right) \omega_Y.
\label{caf2}
\end{equation}

For  an explanation of why the more traditional coarea formula, \cite[Thm. 3.2.11]{Fed} or  \cite[Thm. 5.3.9]{KP},   implies (\ref{caf1}) and (\ref{caf2}) we refer to \cite[Cor. 2.11]{Ncoarea}.

\subsection{The  double fibration trick}  Consider the incidence set
 \[
 \eX:=\bigl\{(\bu, \bx )\in\bsU\times  M;\;\;\bu(x)=0\,\bigr\}.
 \]
 It comes equipped with two natural projections
 \[
 \bsU\stackrel{\pi_-}{\longleftarrow}\eX\stackrel{\pi_+}{\longrightarrow} M,
 \]
 \[
 \pi_+(\bu,\bx)=\bx,\;\;\pi_-(\bu,\bx)=\bu,\;\;\forall (\bx,\bu)\in\eX.
 \]
For any subset $A\subset M$ and $B\subset \bsU$ we set
\[
\eX^+_A:=\pi_+^{-1}(A),\;\;\eX^-_B:=\pi_-^{-1}(B).
\]

 \begin{lemma} (a) The incidence set  $\eX$ has a natural structure of   smooth manifold diffeomorphic to the total space of the vector bundle $K\to M$.
 
 \noindent (b)   If $\bu\neq 0$ is a   regular value of $\pi_-$, then $\bu\pitchfork 0$.
 \label{lemma: sard}
 \end{lemma}
 
 \begin{proof}  (a) Note that
 \[
 (\bx,\bu)\in \eX\Llra P_{\bx}\bu=\ev_x \bu=0 \Llra \bu=K_\bx.
 \]
 This proves the first claim.

(b)  Suppose  that $\bu_0\in \bsU\setminus 0$ is a regular value of $\beta$. We will show that for any $\bx_0\in M$ such that $\bu_0(\bx_0)=0$, the adjunction map $\fa_{\bu_0}$  defines an isomorphisn
\[
(T_{Z_{\bu_0}}M)_{\bx_0}\to E_{\bx_0}.
\]
Fix a small open coordinate  neighborhood  $\eO\subset M$ of  $\bx_0$ in $M$ with locall coordinates $(x^1,\dotsc, x^m)$.  We assume that  via these coordinates $\eO$ is identified with a ball $B\subset \bR^m$  centered at $0$ and $\bx_0$ is identified with the center of the ball, $x^i(\bx_0)$, $\forall i=1,\dotsc, m$.

 Both bundles $E$ and $K$ are trivializable over  $B$.  We can therefore find smooth maps
 \[
 \be_1,\dotsc, \be_N:  \eO\to \bsU
 \]
 such that the following hold.
 
 \begin{gather}
 \mbox{For any $\bx\in \eO$ the collection  $\{\be_a(\bx)\}_{1\leq a\leq N}$  is an orthonormal  basis of $\bsU$.}\\
\mbox{ $\spa\bigl\{\,\be_i(\bx),\;\; 1\leq i\leq r\,\bigr\}=E_\bx$, $\forall \bx\in\eO$.}\\
 \mbox{$\spa\bigl\{\,\be_\alpha(\bx),\;\; r< \alpha\leq N\,\bigr\}=K_\bx$,  $\forall \bx\in\eO$.}\label{26}\\
 \mbox{$\nabla^E \be_i(\bx_0)=0,\;\;\forall i=1,\dotsc, r$,} \label{frame}
 \end{gather}

 We will use the following  conventions frequently encountered in integral geometry.
 
 \begin{itemize}
 
  \item We will use the Latin letters $a,b,c$ to denote indices in the range $1,\dotsc, N$.

 \item We will use the Latin  letters $i,j,k,\ell$ to  denote indices in the range $1,\dotsc, r={\rm rank}\,(E)$.
 
 \item We will use the Greek letters $\alpha,\beta,\gamma$ to denote indices in the range $r+1,\dotsc, N$.

 \end{itemize}
 
 The map
 \[
 \bR^N\times B\ni (t, x)\mapsto \left(\sum_{a}t^a\be_a(x), x\right)\in \bsU\times \eO
 \]
 is a diffeomorphism. The  set $\eX^+_{\eO}\subset \underline{\bsU}_\eO$  can be  identified with the set 
\begin{equation}\label{ttaux}
\bigl\{ ( t^1,\dotsc, t^N, \underbrace{x^1,\dotsc, x^m}_x\,)\in \bR^N\times \bR^m;\;\; x\in B,\;\;t^j=0,\;\;\forall j\leq r\,\bigr\}.
\end{equation}
We write
\begin{equation}\label{ttau}
t:=(t^i)_{1\leq i\leq r},\;\;\tau:=(t^{\alpha})_{r<\alpha \leq N},\;\;\tilde{t}:=(t,\tau).
\end{equation}
Thus the pair $(\tau,x)$ defines local coordinates on $\eX^+_\eO$. In these coordinates the pair $(\bu_0,\bx_0)$ is identified with a pair $(\tau_0, 0)\in\bR^{N-r}\times \bR^m$,
\[
\tau_0=(\tau_0^{r+1},\dotsc,\tau_0^N). 
\]
Moreover,  the   map $\pi_-$ is  given by
\begin{equation}\label{pi-}
(\tau, x)\mapsto\pi_-(\tau, x)=\sum_{\alpha}t^\alpha\be_\alpha(x)\in\bsU.
\end{equation}
We set
\[
u^a(x):=\bigl(\, \bu_0, \be_a(x)\,\bigr)_\bsU,\;\;\forall a=1,\dotsc, N, 
\]
so that
\begin{equation}
\bu_0=\sum_a u^a(x)\be_a(x),\;\;\forall x\in B.
\label{u0}
\end{equation}
Above, we think of $\bu_0$ as  a constant section of the trivial bundle $\underline{\bsU}_M$. The functions $u^a(x)$ are the coordinates of this section in the moving frame $(\be_a(x))$. Note that
\begin{equation}
S_{\bu_0}^E(x)=P_x\bu_0=\sum_i u^i(x)\be_i(x).
\label{seu0}
\end{equation}
The fiber $\eX^{-}_{\bu_0}=\pi_-^{-1}(\bu_0)$   is    described in the coordinates $(\tau,x)$ by the equalities
\[
u^i(x)=0,\;\;t^\alpha=u^\alpha(x),\;\;\forall 1\leq i\leq r, \;\;\forall \alpha>r.
\]
\begin{remark}\label{rem: Q} Denote by $Q$ the natural orthogonal projection $Q:\underline{\bsU}_M\to K=\ker P$. From the above equalities and (\ref{26}) we deduce  that the  section
\begin{equation}
Q\bu_0: M\to K,\;\;x\mapsto Q_x\bu_0,
\label{Z}
\end{equation}
induces a homeomorphism  from $Z_{\bu_0}$ to the  fiber $\eX^{-}_{\bu_0}$.  This homeomorphism would be a diffeomorphism if $Z_{\bu_0}$ were cut out transversally by the the equations $u^i(x)=0$, $1\leq i\leq r$.\qed
\end{remark}

The differential of $\pi_-$   at $(\tau_0,0)\in\eX^-_{\bu_0}$  is
\[
d\pi_-|_{\tau_0,0}= \sum_\alpha dt^\alpha\be_\alpha|_{\tau=\tau_0} +\sum_\alpha \tau_0^\alpha  d\be_\alpha|_{x=0}.
\]
Since $\bu_0$ is a regular value of $\pi_-$,  the differential $d\pi_-$ at any point in $\eX^-_{\bu_0}$  is surjective.  In particular, the  induced linear map
\[
Pd\pi_-|_{\tau_0,0} =\sum_\alpha \tau_0^\alpha P d\be_\alpha(x)|_{\bx=0}: T_{\bx_0}M\to E_{\bx_0}
\]
must be surjective.  From (\ref{seu0}) we deduce that 
\[
\nabla^ES^E_{\bu_0}= P d\left(\,\sum_i u^i(x)\be_i(x)\,\right)= \sum_i  du^i \be_i+  \sum_i u^i P d\be_i 
\]
At $\bx_0$ we have $u^i(\bx_0)=0$  and we conclude that
\[
\bigl(\, \nabla^ES^E_{\bu_0}\,\bigr)|_{\bx_0}=\sum_i  du^i \be_i.
\]
On the other hand,  we deduce from (\ref{u0}) that
\[
0=d\left(\sum_a u^a(x)\be_a(x)\right)\Ra Pd\left(\sum_a u^a(x)\be_a(x)\right)=0
\]
\[
\Ra \sum_i  du^i \be_i+  \sum_i u^i P d\be_i=-\sum_\alpha u^\alpha P d\be_\alpha.
\]
At $\bx_0$  we have $u^i(\bx_0)=0$, $u^\alpha(\bx_0)=\tau_0^\alpha$ and we deduce
\[
\bigl(\, \nabla^E S_{\bu_0}^E\,\bigr)|_{\bx_0}=\sum_i  du^i \be_i= -\sum_\alpha \tau_0^\alpha P d\be_\alpha(x)|_{x=0}=-Pd\pi_-|_{\tau_0,0}.
\]
This proves that the  adjunction map 
\begin{equation}
\fa_{\bu_0}|_{\bx_0}=\bigl(\, \nabla^E S_{\bu_0}^E\,\bigr)|_{\bx_0}=-Pd\pi_-|_{\tau_0,0}
:  T_{\bx_0}M\to E_{\bx_0}
\label{adj}
\end{equation}
is surjective.  Since
\[
 \sum_i  du^i \be_i=-Pd\pi_-|_{\tau_0,0}
 \]
 we deduce that near $\bx_0$ the zero set $Z_{\bu_0}$ is cut out transversally  by the equations $u^i(x)=0$, $i=1,\dotsc, r$.
 \end{proof}
 
 Observe that it suffices to prove  (\ref{intfor})   only for  forms $\eta$ supported in some coordinate  neighborhood $\eO$ of some point $\bx_0 \in M$.  We continue to use the notations and the conventions introduced in the proof of Lemma \ref{lemma: sard}.    We have a double fibration    
 \[
 \bsU\stackrel{\pi_-}{\longleftarrow}\eX|_{\eO}\stackrel{\pi_+}{\longrightarrow} \eO.
 \]
Assume that  the volume form 
\[
\omega_\eO=dx^1\wedge\cdots \wedge dx^m\in\Omega^m(\eO)
\]
defines the given orientation of $M$. Clearly, the  equality (\ref{intfor}) is linear in  $\eta$ so it suffices to  prove   it in the special case when
\[
\eta=f_M dx^{r+1}\wedge \cdots \wedge dx^m,\;\;f_M\in C_0^\infty(\eO).
\]
 We fix an orientation on $\bsU$  and consider the volume form
 \[
 \omega_\bsU= \rho_\bsU dV_\bsU,\;\;\rho_\bsU=\frac{1}{(2\pi)^{\frac{N}{2}}}e^{-\frac{|\bu|^2}{2}},
 \]
  where $dV_\bsU$ denotes the Euclidean volume form on $\bsU$   determined by the given orientation.   
  
  The  orientation on $\bsU$  defines an orientation on the trivial bundle $\underline{\bsU}_M$. Coupled with the orientation  on $E$  it induces  an orientation on the vector  bundle  $K$ uniquely determined  by the requirements
  \[
  {\rm orientation}\,(\underline{\bsU}_M)={\rm orientation}\,(E)\wedge {\rm orientation}\,(K)={\rm orientation}\,(K)\wedge {\rm orientation}\,(E).
  \]
  Finally,  the  orientation on  $K$ induces an orientation on the total space $\eX$ via  the fiber-first  convention. We will refer to  this orientation as the \emph{natural orientation} on $\eX$.
  
  For any  regular  value $\bu_0$ of $\pi_-$, the fiber $\eX^-_{\bu_0}$ caries an orientation  given by the fiber-first convention applied to the fibration $\pi_-:\eX\to\bsU$.
  
  \begin{lemma} The natural  orientation  of $\eX|_\eO$   has the property that for any   regular value $\bu_0$ of  $\pi_-$,  the natural diffeomorphism
  \[
 Q\bu_0: Z_{\bu_0}  \to \eX^-_{\bu_0} 
  \]
  defined in Remark \ref{rem: Q}  has degree  $(-1)^{Nm}$ and thus changes the orientation by the factor $(-1)^{Nm}$.
  \label{orient1}
  \end{lemma}
  
  \begin{proof} The fiber  $\eX^-_{\bu_0}$ is  the image of $Z_{\bu_0}$   via  the section $\Psi=Q\bu_0$  of $\eX\to M$.    The map $\Psi$  identifies the normal bundle  $T_{Z_{\bu_0}}M$ of $Z_{\bu_0}$ in $M$  with the normal bundle $T_{\eX^-_{\bu_0}}\Psi(M)$ of $\eX^-_{\bu_0}$ in $\Psi(M)$. 
  
  The equality (\ref{adj}) shows that  the restriction of $d\pi_-$ to $T_{\eX^-_{\bu_0}}\Psi(\eO)$  can be identified up to a sign with the opposite of the adjunction map. This sign is not important  for orientations purposes since the bundles involved   have even rank.  Now observe that at $(\bu_0,\bx_0)\in \eX$ we have
  \[
  {\rm orientation}\, (\eX)= {\rm orientation}\,(K_{\bx_0})\wedge  {\rm orientation}\,\Psi(M)
  \]
  \[
  = {\rm orientation}\,(K_{\bx_0})\wedge  {\rm orientation}\,(Z_{\bu_0})\wedge  {\rm orientation}\,(E_{\bx_0})
  \]
  \[
  =(-1)^{Nm} {\rm orientation}\,(Z_{\bu_0}) \wedge  {\rm orientation}\,(E_{\bx_0})\wedge  {\rm orientation}\,(K_{\bx_0}).
  \]
  On the other hand
  \[
   {\rm orientation}\,(\eX)= {\rm orientation}\,(\eX^-_{\bu_0})\wedge  {\rm orientation}\,\bsU
   \]
   \[
   = {\rm orientation}\,(\eX^-_{\bu_0})\wedge  {\rm orientation}\,(E_{\bx_0})\wedge  {\rm orientation}\,(K_{\bx_0}).
   \]
    \end{proof}

  The first coarea formula (\ref{caf1})  coupled with  Lemma \ref{orient1} imply  that
 
  \[
  \int_\bsU \left(\int_{Z_u}\eta\right) \rho_\bsU dV_\bsU=(-1)^{Nm}  \int_\bsU \left(\int_{\eX^-_\bu}\eta\right) \rho_\bsU dV_\bsU=(-1)^{Nm}\int_{\eX^+_{\eO}} \pi_+^*\eta \wedge \pi^*_-\omega_\bsU.
  \]
  Hence
   \begin{equation}
    \int_\bsU \left(\int_{Z_u}\eta\right) \rho_\bsU dV_\bsU= \int_{\eX^+_{\eO}} \pi^*_-\omega_\bsU\wedge  \pi_+^*\eta. 
  \label{caf3}
  \end{equation}
  Recalling that $\pi_+^{-1}(x)=K_x$, $\forall x\in \eO$, we deduce from (\ref{caf3}) and the second coarea formula   (\ref{caf2})  that 
  \begin{equation}
  \int_\bsU \left(\int_{Z_u}\eta\right) \rho_\bsU dV_\bsU=\int_{\eO} \left( \int_{K_x}\frac{ \pi^*_-\omega_\bsU\wedge \pi_+^*\eta }{\pi_+^*\omega_\eO}\right)\omega_\eO.
  \label{intfor1}
  \end{equation}
This is  Gelfand's double fibration trick. To prove  (\ref{intfor}) we need to show that
 \begin{equation}
 \left( \int_{K_x}\frac{\pi^*_-\omega_\bsU\wedge \pi_+^*\eta }{\pi_+^*\omega_\eO}\right)\omega_\eO=\frac{1}{(2\pi)^h} \eta \wedge \pf\bigl(-F^E\bigr)=\frac{1}{(2\pi)^h} \pf\bigl(-F^E\bigr)\wedge \eta \;\;\mbox{on} \;\;\eO.
 \label{intfor2}
 \end{equation}

\subsection{Proof of (\ref{intfor2})} Suppose  that $\bigl(\,\be_a(0)\,\bigr)_{1\leq a\leq N}$ is a positively oriented basis of $\bsU$ and $(\,\be_i(0)\,)_{1\leq i\leq r}$ is a positively oriented basis of $E_{\bx_0}$.    We set
 \[
 y_{ab}(x):=\bigl(\,\be_a(0), \be_b(x)\,\bigr)_\bsU,\;\;\forall 1\leq a,b\leq N.
 \]
 The  $N\times N$ matrix $Y(x)=(y_{ab}(x))$ is orthogonal and $Y(0)=1$.   Moreover
 \begin{equation}
 \be_a(x)=\sum_b y_{ba}(x)\be_b(0),\;\;\be_a(0)=\sum_b y_{ab}(x)\be_b(x),\;\;\forall a.
 \label{decomp}
 \end{equation}
We deduce 
\[
P_{x}\be_a(0)= \sum_iy_{ai}(x)\be_i(x)=\sum_{i,b}y_{ai}(x)y_{bi}(x)\be_b(0).
\]
Hence
\[
\nabla^E \be_j(x)= P_xd\sum_by_{bj}(x)\be_b(0)=\sum_b dy_{bj}(x)P\be_b(0)=\sum_{i,b}y_{bi}(x) dy_{bj}(x)\be_i(x).
\]
Thus, in the local orthonormal frame $(\,\be_i(x)\,)$ the connection $\nabla^E$ is described by the  matrix-valued $1$-form
\[
\Gamma=(\Gamma_{ij}(x))_{1\leq i,j\leq p},\;\;\Gamma_{ij}(x)=\sum_{b}y_{bi}(x) \wedge dy_{bj}(x).
\]
The curvature  of $\nabla^E$ is $F^E=d\Gamma +\Gamma\wedge \Gamma$. Note that
\[
d\Gamma_{ij}(x)= \sum_{b}d y_{bi}(x) \wedge dy_{bj}(x).
\]
At $\bx_0$, the constraint (\ref{frame}) on the frame $\be_i(x)$ implies that  $\nabla^E\be_j|_{\bx_0}=0$, $\forall j$. Thus
\begin{equation}
0=\Gamma_{ij}(\bx_0)=\sum_{b}y_{bi}(0) dy_{bj}(0)=\sum_{b}\delta_{bi} dy_{bj}(0)=dy_{ij}(0),\;\;\forall i,j.
\label{gammasync}
\end{equation}
Hence 
\[
F^E|_{\bx_0}= d\Gamma= (F_{ij})_{1\leq i,j\leq p},
\]
\[
 F_{ij}=\sum_{b}d y_{bi}(0) \wedge dy_{bj}(0)=\sum_\beta d y_{\beta i}(0) \wedge dy_{\beta j}(0)\in \Lambda^2 T^*_{\bx_0} M.
\]
On the other hand, the  $N\times N$ Maurer-Cartan matrix $Y^{-1}(x)dY(x)$ is skew-symmetric  for any $x$.  At $x=0$  we have $Y(0)=1$  and we deduce 
\[
dy_{\beta i}(0)=-dy_{i\beta}(0),\;\;\forall i,\beta.
\]
We conclude that
\begin{equation}
F^E|_{\bx_0}= d\Gamma= (F_{ij})_{1\leq i,j\leq r},\;\;  F_{ij}=\sum_\beta d y_{i \beta }(0) \wedge dy_{j \beta}(0).
\label{re}
\end{equation}
Define 
\[
 y_a:\bsU\to\bR\;\;, \;\;y_a(\bu)=\bigl(\,\bu,\be_a(0)\,\bigr)_\bsU,\;\;1\leq a\leq N.
 \]
 The Euclidean volume form on $\bsU$ is  then
 \[
 dV_{\bsU}=dy_1\wedge \cdots \wedge dy_N.
 \]
 Recall that $(\tau,x^1,\dotsc, x^m)$ are coordinates on $\eX_\eO^+$; see (\ref{ttaux}) and  (\ref{ttau}). Using (\ref{pi-}) we deduce that
 \[
 y_a\bigl(\,\pi_-(\tau, x^1,\dotsc, x^m )\,\bigr)= y_a\left(\sum_\alpha t^\alpha \be_\alpha(x)\,\right)= \sum_\alpha t^\alpha\bigl(\, \be_a(0),\, \be_\alpha(x)\,\,\bigr)_\bsU=\sum_\alpha t^\alpha y_{a\alpha}(x).
 \]
We set
 \[
 \xi_a(x)=\xi_a(\tau,x):=\sum_\alpha t^\alpha y_{a\alpha}(x),
 \]
 so that
 \[
 \pi_-(\tau,x)=\sum_a\xi_a(x) \be_A(0),
 \]
and
 \[
 \pi_-^*dV_{\bsU}=d\xi_1\wedge \cdots \wedge d\xi_N.
 \]
We view this as a form on the space $\bR^{N-r}\times \eO$ with coordinates $(\tau,x)$. We have
  \[
  d\xi_a= \sum_\alpha dt^\alpha y_{a\alpha}+\sum_\alpha t^\alpha dy_{a\alpha} (x).
  \]
  Observe that at $(\tau_0,0)$ we have
  \[
  y_{ab}(0)=\delta_{ab},\;\;t^\alpha=\tau_0^\alpha,
  \]
  so 
  \[
  d\xi_a(0):=d\xi_a|_{x=0}= \sum_\alpha\delta^a_\alpha dt^\alpha + \sum_\alpha \tau_0^\alpha dy^a_\alpha(0).
  \]
  Hence
  \[
  d\xi_i(0)=\sum_\alpha \tau_0^\alpha dy_{i\alpha}(0),\;\;d\xi_\beta= dt^\beta +\sum_\alpha\tau_0^\alpha dy_{\beta\alpha}(0),
  \]
  so that
  \[
 \pi_-^*\omega_\bsU= \frac{1}{(2\pi)^{\frac{N}{2}} }e^{-\frac{|\tau|^2}{2}} d\xi_1\wedge \cdots d\xi_N.
 \]
Now observe that 
  \[
  d\xi_1\wedge \cdots d\xi_N=\underbrace{(dt^{r+1}\wedge \cdots \wedge dt^N)}_{=:d\tau}\,\underbrace{\bigwedge_i \sum_{\alpha_i} \tau_0^{\alpha_i} dy_{i\alpha_i}(0)}_{=:\Omega(\tau_0)}\,+\, \eL,
  \]
 where $\eL$   incorporates all  the other terms  that have degrees $<N-r$ in the $dt^\alpha$ variables, and
 \[
 \Omega(\tau_0)\in \Lambda^r T^*_{\bx_0}M.
 \]
 Since the  terms collected in $\eL$ have degrees $>r$ in the variables $(x^1,\dotsc, x^m)$ we deduce
 \[
 d\xi_1\wedge \cdots d\xi_N\wedge \pi_+^*\eta=f_Md\tau\wedge \Omega(\tau_0)\wedge dx^{r+1}\wedge\cdots \wedge dx^m.
 \]
 Denote by $\Omega(\tau_0)_{1,\dotsc,r}$ the coefficient  of $dx^1\wedge \cdots \wedge dx^r$ in the decomposition of $\Omega(\tau_0)$  with respect to the basis $\{\,dx^{j_1}\wedge \cdots \wedge dx^{j_r}\,\}_{1\leq j_1<\cdots ,j_r\leq m}$ of $\Lambda^r T^*_{\bx_0} M$. If we set
 \[
 \gamma_K(d\tau):= \frac{1}{(2\pi)^{\frac{N-r}{2}}} e^{-\frac{|\tau|^2}{2}} d\tau\in \Omega^{N-r}(K_{\bx_0}),
 \] 
 then we deduce that
 \begin{equation}
 \frac{\pi_-^*\omega_\bsU\wedge \pi_+^*\eta}{dx^1\wedge \cdots \wedge dx^m}=\frac{1}{(2\pi)^{\frac{p}{2}}}\gamma_K \wedge f_M(\bx_0)\Omega(\tau_0)_{1,\dotsc,r}.
 \label{resid}
 \end{equation}
 Hence
 \begin{equation}
 \int_{K_{\bx_0}} \frac{\pi_-^*\omega_\bsU\wedge \pi_+^*\eta}{dx^1\wedge \cdots \wedge dx^m}=\frac{f_M(\bx_0)}{(2\pi)^{\frac{r}{2}}}\int_{K_{\bx_0}}\Omega(\tau)_{1,\dotsc,r} \gamma_K(d\tau).
 \label{intresid}
 \end{equation}
 In the sequel we will denote by $\bullet$ the  inner product in the space $K_{\bx_0}$ Our choice of local frames   amounts to a metric isomorphism $K_{\bx_0}\cong\bR^{N-r}$. 
 
 For every $i=1,\dotsc, r$  and $\tau\in K_{\bx_0}$ we   set
 \[
 \Phi_i:=\left[
 \begin{array}{c} 
 dy_{i r+1}(0)\\
 \vdots\\
 dy_{iN}(0)
 \end{array}
 \right]\in   T^*_{\bx_0}M\otimes K_{\bx_0},\;\; \omega_i(\tau)= \Phi_i\bullet\tau:=\sum_\alpha t^\alpha dy_{i\alpha}(0)\in    T^*_{\bx_0}M.
 \]
 Let us point out that the $(N-r)\times r$ matrix  with columns $\Phi_1,\dotsc, \Phi_r$  describes the differential at $\bx_0$ of the Gauss map 
 \[
 M\ni\bx\mapsto E_\bx\in \Gr_r(\bsU)=\mbox{the Grassmannian of $r$-planes in $\bsU$}.
 \]
 We have 
 \[
 \Omega(\tau)=\omega_1(\tau)\wedge \cdots \wedge \omega_r(\tau).
 \]
 For every $j=1,\dotsc, m$ and $\tau\in K_{\bx_0}$  we set
 \[
 \Phi_{ij}:=\pa_{x^j}\inpr\Phi_i= \left[
 \begin{array}{c} 
 \frac{\pa y_{i r+1}}{\pa x^j}(0)\\
 \vdots\\
 \frac{\pa y_{iN}}{\pa x^j} (0)
 \end{array}
 \right]\in K_{\bx_0},\;\;\omega_{ij}(\tau)=(\Phi_{ij},\tau)_\bsU=\Phi_{ij}\bullet\tau\in\bR.
 \]
 We denote by $A(\tau)$ the $r\times r$ matrix with entries
 \[
 A(\tau)_{ij}=\omega_{ij}(\tau),\;\;1\leq i,j\leq r.
 \]
Then
\[
\omega_i(\tau)=\sum_{j=1}^m \omega_{ij}(\tau)dx^j,\;\;\forall i=1,\dotsc, r, \;\;\Omega(\tau)_{1,\dotsc,r}=\det A(\tau).
\]
We  set
\begin{equation}
\overline{\Omega}_{1,\dotsc, r}:=\int_{K_{\bx_0}}\det A(\tau)\gamma_K(d\tau).
\label{average}
\end{equation} 
Using (\ref{intresid}) we deduce
\begin{equation}
\int_{K_{\bx_0}} \frac{\pi_-^*\omega_\bsU\wedge \pi_+^*\eta}{dx^1\wedge \cdots \wedge dx^m}=\frac{f_M(\bx_0)}{(2\pi)^{\frac{p}{2}}}\overline{\Omega}_{1,\dotsc, r}.
\label{inresid1}
\end{equation}
To compute  the Gaussian average (\ref{average})  we  use the theory of orthogonal invariants \cite{Weinv}  as in   Weyl's  proof of  his  tube  formula \cite[\S4.4]{Gray}, \cite[\S 9.3.3]{N1}, \cite{Wetube}.

Let us  first observe that for $1\leq i_1\neq i_2\leq r$  and $1\leq j_1<j_2\leq m$ we have
 \[
 \Phi_{i_1j_1}\bullet \Phi_{i_2j_2}-  \Phi_{i_1j_2}\bullet \Phi_{i_2j_1}=\sum_\alpha \left(\frac{\pa y_{i_1\alpha}}{\pa x^{j_1}} \frac{\pa y_{i_2\alpha}}{\pa x^{j_2}}-\frac{\pa y_{i_1\alpha}}{\pa x^{j_2}} \frac{\pa y_{i_2\alpha}}{\pa x^{j_1}}\right)
 \]
 \[
= \left( \sum_\alpha dy_{i_1\alpha}\wedge dy_{i_2\alpha}\right)(\pa_{x^{j_1}},\pa_{x^{j_2}} ).
 \]
 Using (\ref{re}) and the notation (\ref{not})  we deduce
 \begin{equation}
 F^{E}_{i_1i_2|j_1j_2}= \Phi_{i_1j_1}\bullet\Phi_{i_2j_2}-  \Phi_{i_1j_2}\bullet \Phi_{i_2j_1},\;\;\forall 1\leq i_1, i_2\leq r,\;\;1\leq j_1,j_2\leq m.
 \label{re1}
 \end{equation}
 For   any collection  of vectors $\bu_{ij}\in K_{\bx_0}$, $1\leq i,j\leq r$ and any $\tau\in K_{\bx_0}$  we  define the $r\times r$ matrix
 \[
 A(\tau,\bu_{ij}):= \bigl(\,\bu_{ij}\bullet \tau\,\bigr)_{1\leq i,j\leq r}, 
 \]
and we  consider  the average
 \[
\mu(\bu_{ij}):= \int_{K_{\bx_0}} \det A(\tau,\bu_{ij}) \gamma_K(d\tau).
 \]
 The average $\mu(\bu_{ij})$ is   a polynomial in the variables $\bu_{ij}\in K_{\bx_0}$, $1\leq i,j,\leq r$, and it is invariant with respect  to the action of the  group $O(N-r)$ of orthogonal transformations of $K_{\bx_0}$. Note that when $\bu_{ij}=\Phi_{ij}$ we have
 \[
 \mu(\Phi_{ij})=\overline{\Omega}_{1,\dotsc, r}.
 \]

  We recall that  $r=2h$ and we denote by $\eS_r=\eS_{2h}$ the group of permutations of $\{1,2,\dotsc, 2h\}$.  As in \cite[\S 9.3.3]{N1} we define
  \[
 Q_{\si,\vfi}(\bu_{ij}):=\prod_{j=1}^h\bigl(\bu_{\vfi_{2j-1}\si_{2j-1}}\bullet \bu_{\vfi_{2j}\si_{2j}}\bigr),\;\; Q= Q(\bu_{ij}):= \sum_{\si ,\vfi\in\eS_r}\eps(\si\vfi) Q_{\si,\vfi}(\bu_{ij}).
  \]
  Lemma 9.3.9 in \cite{N1} shows that there exists a constant $Z$ such that
  \[
  \mu(\bu_{ij})=ZQ(\bu_{ij}),\;\;\forall \bu_{ij}.
  \]
  To find the constant $Z$  we choose  the variables $\bu_{ij}\in K_{\bx_0}$ judiciously.  More precisely, we set
  \[
  \bu^*_{ij}:=\begin{cases}
  \be_N(0), & i=j,\\
  0, & i\neq 0.
  \end{cases}
  \]
  In this case
  \[
  A(\tau,\bu^*_{ij}) =\diag(\underbrace{t^N,\dotsc, t^N}_{2h}),\;\;\det A(\tau,\bu_{ij})=|t^N|^{2h},
  \]
  \[
  \mu\bigl(\,\bu_{ij}^*\,\bigr)= \int_{K_{\bx_0}}\bigl|\,t^N\,\bigr|^{2h} \gamma_K(d\tau)= \frac{1}{\sqrt{2\pi}}\int_{\bR} s^{2h} e^{-\frac{s^2}{2}} ds=\prod_{j=1}^h (2j-1)=(2h-1)!!.
  \]
 On the other  hand,
 \[
 Q_{\si,\vfi}\bigl(\,\bu_{ij}^*\,\bigr)= \begin{cases} 
 1, &\si=\vfi,\\
 0, &\si\neq \vfi,
 \end{cases}
 \]
 and we deduce that $Q\bigl(\,\bu_{ij}^*\,\bigr)=(2h)!$.  Thus 
 \[
 Z=\frac{(2h-1)!!}{(2h)!}=\frac{1}{2^h h!},\;\; \mu(\Phi_{ij})=\frac{1}{2^h h!}Q(\Phi_{ij}).
 \]
 Denote by $\eS_r'$ the set of permutations $\vfi$ of $\{1,2,\dotsc, 2h\}$ such that
 \[
 \vfi_1<\vfi_2,\;\;\vfi_3<\vfi_4,\dotsc, \vfi_{2h-1}<\vfi_{2h}.
 \]
 Using (\ref{re1})  we   deduce as in the proof  of \cite[Eq. (9.3.11)]{N1} that
 \begin{equation}
 Q(\Phi_{ij})=2^h\sum_{\si,\vfi\in\eS_r'} \;\prod_{j=1}^h\eps(\si\vfi)F^E_{\vfi_{2j-1}\vfi_{2j}|\si_{2j-1}\si_{2j}}.
 \label{qre}
 \end{equation}
 Thus
 \[
 \mu(\Phi_{ij})=\frac{1}{h!} \sum_{\si,\vfi\in\eS_r'} \prod_{j=1}^h\eps(\si\vfi)F^E_{\vfi_{2j-1}\vfi_{2j}|\si_{2j-1}\si_{2j}}\stackrel{(\ref{pf2})}{=}\pf\bigl(-F^E\bigr)(\pa_{x^1},\cdots,\pa_{x^r}).
 \]
 \begin{equation}
 \overline{\Omega}_{1,\dotsc,r}=\mu(\Phi_{ij})=\pf\bigl(-F^{E}\bigr)(\pa_{x^1},\dotsc,\pa_{x^r}).
 \label{average1}
 \end{equation}
 Using (\ref{inresid1}) and (\ref{average1}) we conclude that
 \[
 \left(\int_{K_{\bx_0}} \frac{\pi_-^*\omega_\bsU\wedge \pi_+^*\eta}{dx^1\wedge \cdots \wedge dx^m}\right) dx^1\wedge \cdots \wedge dx^m= \frac{f_M(\bx_0)}{(2\pi)^{\frac{r}{2}}}\overline{\Omega}_{1,\dotsc, p}dx^1\wedge\cdots \wedge dx^m 
 \]
 \[
 =\frac{1}{(2\pi)^h} \pf\bigl(-F^E\bigr)\wedge \eta .
  \]
 This proves (\ref{intfor2}). \qed

 \section{The  white noise limit}
 \label{s: 3}
 \setcounter{equation}{0}
 
 \subsection{Gaussian measures}\label{ss: 41} Recall \cite{Bog}   that a  \emph{centered Gaussian measure} on a  finite dimensional real vector space $\bsU$ is a probability measure $\gamma$ on $\bsU$ such that for any linear functional $\xi\in\bsU^*=\Hom( \bsU,\bR)$  the pushforward $\xi_{\#}\gamma$ is   Gaussian measure on $\bR$
 \[
 \xi_{\#}\gamma=\bgamma_v:=\frac{1}{\sqrt{2\pi v}} e^{-\frac{\xi^2}{2v}} d\xi, \;\;v\geq 0.
 \]
 Above, when $v=0$, we define $\gamma_v$ to  the Dirac  delta-measure concentrated at $0$.
 
 A centered Gaussian measure $\gamma$ on $\bsU$ is completely determined by its covariance form $C=C_\gamma$ which is the symmetric, nonnegative definite   bilinear form
 \[
 C:\bsU^*\times\bsU^*\to\bR,\;\; C(\xi_1,\xi_2)= \bsE_\gamma(\xi_1\cdot\xi_2),
 \]
 where  $\xi_1,\xi_2\in\bsU^*$ are viewed as random variables on $(\bsU, \gamma)$. The   Gaussian measure $\gamma$ is called \emph{nondegenerate} if its covariance form is  nondegenerate.   If this is the case, the bilinear form defines an Euclidean inner product  on $\bsU^*$ and, by duality, an inner product on $\bsU$. 
 
 Conversely, given an inner product $\bsi$ on $\bsU$ with norm $|-|_{\bsi}$, we have a Gaussian measure
 \begin{equation}
\gamma_{\bsi}=(2\pi)^{-\frac{\dim\bsU}{2}} e^{-\frac{|\bu|_{\bsi}^2 }{2} } |d\bu|_{\bsi},
 \label{met_gauss}
 \end{equation}
 and $\bsi$ coincides with  the inner product determined by $\gamma_{\bsi}$. 
 
 The inner product $\bsi$ identifies $\bsU$ with $\bsU^*$ and the covariance form  of an arbitrary  Gaussian measure  $\gamma$ on $\bsU$ can be identified with a symmetric  nonnegative operator $T_\gamma:\bsU\to\bsU$. The measure $\gamma$ is nondegenerate iff $T_\gamma$ is invertible. In this case
 \begin{equation}
 \gamma =\frac{1}{\sqrt{\det 2\pi T_\gamma}} e^{-\frac{1}{2}\bsi(T_\gamma^{-1}\bu,\bu)} |d\bu|_{\bsi}=\Bigl(T_\gamma^{\frac{1}{2}}\Bigr)_\#\gamma_{\bsi}.
 \label{opgau}
 \end{equation}
 Note that if $\gamma$ is a centered Gaussian measure on $\bsU$ with covariance form $C_\gamma$ and $L:\bsU\to\bsV$ is a linear map to another finite dimensional vector space $\bsV$ then the pushforward $L_\#\gamma$ is a Gaussian measure on $\bsV$ with covariance form $C_{L_\#\gamma}=L^*C_\gamma$. In particular, if $\gamma$ is as in (\ref{opgau}), then
 \[
 \gamma=\Bigl(T_\gamma^{\frac{1}{2}}\Bigr)_\#\gamma_{\bsi}.
 \]

 \subsection{Probabilistic descriptions of special metrics and connection}\label{ss: 42}  Suppose that we are  given a smooth  real vector bundle $E\to M$ of rank $r$,  and a sample space $(\bsU,\bgamma_\bsU)$ of $C^\infty(E)$.    The nondegenerate Gaussian measure  $\gamma_\bsU$ on $\bsU$  determines a metric $(-,-)_\bsU$.
 
 As we have seen, the  metric $(-,-)_\bsU$ on $\bsU$ induces a metric $(-,-)_E$ on the bundle $E$ and by duality, a metric on $E^*$.   We want to  give a probabilistic description of the induced metric $(-,-)_{E^*}$ in a fiber $E^*_{\bx}$ of $E^*$.  
 
 To simplify the presentation we introduce some notations and conventions.
 
 \begin{enumerate} 
 
  \item  We will  use the $\bullet$-notation to denote the inner product in $\bsU$ or  $\bsU^*$.

 \item      We will use the Latin letters  $i,j,k,\ell$ to denote indices  in the range $1,\dotsc, m=\dim M$.
 
 \item We will use   the Greek letters $\alpha,\beta,\gamma$ to denote indices in the range $1,\dotsc, r= {\rm rank}\,(E)$.
 
 \end{enumerate}
 Let
 \[
 \lan-,-\ran : E^*_{\bx}\times E_{\bx}\to \bR
 \]
 denote the natural pairing. Fix an orthonormal basis $\Psi_1,\dotsc,\Psi_N$ of $\bsU$ and denote by $(\Psi_n^*)$ the dual orthonormal basis of $\bsU^*$.  Then $\ev_\bx^*:E_\bx^*\to\bsU^*$ is given by
 \[
 \ev_\bx(\bu^*) =\sum_{n=1}^N \bigl\lan\, \bu^*,\Psi_n(x)\,\bigr\ran\Psi_n^*,
 \]
 and
 \[
 (\bu_1^*,\bu_2^*)_{E^*}= (\ev_\bx^*\bu_1^*)\bullet(\ev_\bx^*\bu_2^*)=\sum_{n=1}^N  \bigl\lan\, \bu_1^*,\Psi_n(x)\,\bigr\ran \bigl\lan\, \bu_2^*,\Psi_n(x)\,\bigr\ran.
 \]
 Thus the  metric $(-,-)_{E^*}$ is described by the   bilinear form $\bsC(\bx)$ on $E^*_\bx$ given by
 \[
 \bsC_\bx=\sum_{n=1}^N \Psi_n(\bx)\otimes \Psi_n(\bx)\in  E_\bx\otimes E_\bx\cong \Hom(E_\bx^*\otimes E_\bx^*,\bR).
 \]
The bilinear form $\bsC_\bx$  has a probabilistic interpretation:  it is the covariance form  of  the  Gaussian measure $(\ev_\bx)_\#\gamma_\bsU$   on $E_\bx$. 
 
 We have a  metric  duality isomorphism
 \[
 \bsD=\bsD_\bx: E_\bx\to E_\bx^*,\;\; (\bv^*,\bsD\bu)_{E^*}:=\lan \bv^*,\bu\ran.
 \]
Fix a point $\bx_0$ and   a  small    coordinate neighborhood $\eO$ of $\bx_0$  with coordinates $(x^i)$ such  that $x^i(\bx_0)=0$.  Suppose  that $(\be^\alpha(x))$ is a local frame  of $E^*$ defined on  $\eO$.  Denote by $(\be_\alpha(x))$ the dual  moving frame. We set
 \[
 C^{\alpha\beta}(x):= \bsC_x\bigl(\,\be^\alpha(x),\be^\beta(x)\,\bigr).
 \]
 The matrix $C(x)=(C^{\alpha\beta}(x))$ is symmetric and positive definite.  We denote by  $(C_{\alpha\beta}(x) )$ the inverse matrix.   If we write
 \[
 \bsD\be_\alpha=:\sum_\beta D_{\beta\alpha}\be^\beta,
 \]
 then we deduce
 \[
\delta^\gamma_\alpha=\lan\be^\gamma,\be_\alpha\ran= \Bigl(\, \be^\gamma, \sum_\beta D_{\beta\alpha}\be^\beta\,\Bigr)_{E^*}=\sum_\beta C^{\gamma\beta}D_{\beta\alpha}
\]
which shows that the duality isomorphism  $\bsD$ is  represented in these  bases by the inverse of the matrix  $C$, $D_{\beta\alpha}(x)=C_{\beta\alpha}(x)$.

We  want to compute  the covariant derivatives
\[
\nabla^{E^*}_{i} \be^\alpha(0):=\nabla^{E^*}_{\pa_{x^i}} \be^\alpha(0).
\]
We set  
\[
\Psi_n^\alpha(x):= \Bigl\lan\,\be^\alpha(x),\Psi_n(x)\,\Bigr\ran\in \bR,\;\;\forall n=1,\dotsc, N,
\]
 and we deduce
\[
\ev_x^*\be^\alpha(x)=\sum_{n=1}^N\Psi^\alpha_n(x) \Psi^*_n,\;\;\pa_{x^i}\Bigl( \ev_x^*\be^\alpha(x)\Bigr)= \sum_{n=1}^N\pa_{x^i}\Psi^\alpha_n(x) \Psi^*_n.
\]
We denote by $P_x$ the orthogonal projection $\bsU^*\to E_x^*$. Then 
\[
\nabla^{E^*}_i\be^\alpha(x)= P_x \pa_i \Bigl(\ev_x^*\be^\alpha(x) \Bigr)=\bsD_x\left(\sum_n \pa_i\Psi_n^\alpha(x)\Psi_n(x)\right)
\]
\[
=\bsD_x\left(\sum_{n,\beta} \pa_i\Psi_n^\alpha(x)\Psi_n^\beta(x)\be_\beta(x)\right)=\sum_{n,\beta,\gamma}\pa_i\Psi_n^\alpha(x)\Psi_n^\beta(x) C_{\gamma\beta}(x) \be^\gamma(x)
\]
\[
=\sum_\gamma\,\underbrace{\left(\sum_n\sum_{\beta}\pa_i\Psi_n^\alpha(x)\Psi_n^\beta(x) C_{\gamma\beta}(x) \right)}_{=:\Gamma^\alpha_{\gamma|i}(x)}\,\be^\gamma(x).
\]
For every $\bx,\by\in \eO$, we  denote by $(x^i)$ the coordinates of $\bx$, by $(y^i)$ the coordinates of $\by$, and we set
\begin{equation}
\begin{split}
C_{\bx,\by}:=\sum_{n=1}^N \Psi_n(\bx)\otimes \Psi_n(\by)\in E_\bx\otimes  E_{\by},\\
C^{\alpha\beta}(x,y):=\sum_{n=1}^N \bigl\lan\,\be^\alpha(\bx),\Psi_n(\bx)\,\bigr\ran\bigl\lan\,\be^\beta(\by),\Psi_n(\by)\,\bigr\ran.
\end{split}
\label{covariance}
\end{equation}
One should think  of $C_{\bx,\by}$ as a \emph{covariance kernel}  defined by the random  section $\bu\in\bsU$ because  it  captures the  correlations between the values  of $\bu$ at $\bx$ and $\by$. We deduce that
\[
\sum_{n}\pa_i\Psi_n^\alpha(x)\Psi_n^\beta(x) =\pa_{x^i}C^{\alpha\beta}(x,y)|_{x=y}.
\]
Hence
\begin{equation}
\nabla_i^{E^*}\be^\alpha(x)=\sum_\gamma\Gamma^\alpha_{\gamma|i}(x)e^\gamma(x),\;\;\Gamma^\alpha_{\gamma|_i}(x)=\sum_\beta \pa_{x^i}C^{\alpha\beta}(x,y)|_{x=y}C_{\gamma\beta}(x).
\label{gamma1}
\end{equation}
By duality we deduce
\begin{equation}
\nabla^E_i\be_\alpha(x) =-\sum_\beta\Gamma^\beta_{\alpha|i}(x)\be_\beta(x).
\label{nablae1}
\end{equation}
We denote by  $\Gamma_i(x)$ the endomorphism  of $E_\bx$ given by
\[
\be_\alpha(x)\mapsto \sum_\beta\Gamma^\beta_{\alpha|i}(x)\be_\beta(x).
\]
From (\ref{gamma1})  and the symmetry of the bilinear form $C(x)$ we deduce that
\begin{equation}
\Gamma_i(x)=\pa_{x^i}C(x,y)|_{x=y}\cdot (C(x)^T)^{-1}=\pa_{x^i}C(x,y)|_{x=y}\cdot C(x)^{-1}.
\label{gamma2}
\end{equation}
We set
\[
\Gamma=\sum_idx^i\Gamma_i =d_xC(x,y)|_{x=y} C(x)^{-1}
\]
The operator valued $1$-form $-\Gamma$   describes the connection $\nabla^E$ in the  local frame $(\be_\alpha(x))$, 
\[
\nabla^E=d-\Gamma. 
\]
The curvature is then
\begin{equation}
F^E=-d\Gamma+\Gamma\wedge \Gamma=-\sum_{i,j}(\pa_{x^i}\Gamma_j-\pa_{x^j}\Gamma_i)dx^i\wedge dx^j+\sum_{i<j}[\Gamma_i,\Gamma_j] dx^i\wedge dx^i.
\label{re2}
\end{equation}
Concretely
\[
\pa_{x^i}\Gamma^\alpha_{\gamma|j}=\pa_{x^i}\sum_n\sum_{\beta}\pa_{x^j}\Psi_n^\alpha(x)\Psi_n^\beta(x) C_{\gamma\beta}(x)
\]
\[
= \sum_n\sum_{\beta}\pa^2_{x^ix^j}\Psi_n^\alpha(x)\Psi_n^\beta(x) C_{\gamma\beta}(x)+\sum_n+\sum_{\beta}\pa_{x^j}\Psi_n^\alpha(x)\pa_{x^i}\Psi_n^\beta(x) C_{\gamma\beta}(x)
\]
\[
\sum_n\sum_{\beta}\pa_{x^j}\Psi_n^\alpha(x)\Psi_n^\beta(x) \pa_{x^i}C_{\gamma\beta}(x).
\]
We deduce
\begin{equation}
\begin{split}
\pa_{x^i}\Gamma_j &=\pa^2_{x^ix^j}C(x,y)|_{x=y}C(x)^{-1}+\pa^2_{x^jy^i}C(x,y)|_{x=y} C(x)^{-1} \\
& +\Bigl(\,\pa_{x^j}C(x,y)|_{x=y}\,\Bigr)\cdot \pa_{x^i}\bigl(\,C(x)^{-1}\,\bigr).
\end{split}
\label{reij}
\end{equation}

Suppose that $E$   came equipped  with another metric $\bsi_0(-,-)$ and connection $\nabla^0$ compatible with this metric.  Then
\[
\nabla^E= \nabla^0+ A=\nabla^0+\sum dx^i\wedge A_i,
\]
where $A$ is a \emph{globally defined}  operator valued $1$-form, $A\in \Omega^1\bigl(\,\End (E)\,)$.

If we choose the local frame frame $(\be^\alpha(x))$ on $\eO$ to be orthonormal with respect to the metric $\bsi_0$, and  $\nabla^0\be^\alpha|_{x=0}=0$, then
 \[
 \pa_i\ev_\bx^*\be^\alpha(x)|_{x=0}=\sum_{n=1}^N \pa_i\bsi_0\bigl(\,\Psi_n(0),\be_\alpha(0)\,\bigr)\Psi_n=\sum_{n=1}^N \bsi_0\bigl(\,\nabla^0_i\Psi_n(0),\be_\alpha(0)\,\bigr)\Psi_n.
\]
It follows that
\begin{equation}
\nabla_i{E^*}\be^\alpha(0)=\sum_\gamma\,\underbrace{\left(\sum_n\sum_{\beta}(\nabla_i^0\Psi_n)^\alpha(0)\Psi_n^\beta(0) C_{\gamma\beta}(0,0) \right)}_{=:A^\alpha_{\gamma|i}(0)}\,\be^\gamma(0),
\label{a_00}
\end{equation}
where
\[
(\nabla^0\Psi_n)^\alpha(x):= \lan\be^\alpha(x),\nabla^0_i\Psi_n(x)\ran.
\]
We deduce
\begin{equation}
\nabla^E_i\be_\gamma(0)=-\sum_\alpha A^\alpha_{\gamma|i}(0) dx.
\label{a_0}
\end{equation}
We  denote by $(A_i(x))$ the endomorphism of $E_x$ given by the matrix $(-A^\alpha_{\gamma|i})_{1\leq \alpha,\gamma\leq r}$.

We can rewrite this in  an invariant way as follows. Consider the  natural projections
\[
M\stackrel{p_+}{\leftarrow} M\times M\stackrel{p_-}{\to} M,\;\;p_\pm(\bx_+,\bx_-)=\bx_\pm,
\]
and    the bundle
\[
E\boxtimes E:=p_+^*E\otimes p_-^* E.
\]
Then $C(\bx_+,\bx_-)$ is a global section of  $E\boxtimes E$.  Its restriction to the diagonal can be identified with the section $C(\bx)$ of the bundle $E\otimes E$ over $M$.  We deduce\begin{equation}
A(\bx) =\sum_i A_i(x) dx^i =-\sum_i\nabla^0_{x^i} C(x,y)_{x=y}  \,\cdot \,C(x)^{-1}.
\label{a}
\end{equation}
Indeed, both sides of the above equality are  globally defined  $\End(E)$-valued $1$-forms on $M$.   It therefore suffices to verify  (\ref{a}) at an arbitrary point $\bx_0$ in  some local coordinates  near $\bx_0$ and some  local trivialization of $E$.   We have done this already in (\ref{a_0}).

We denote by $F^0$ the curvature of $\nabla^0$ and by  $F^E$ the curvature of $\nabla^{E}$.  Then
\[
F^0=\sum_{i<j}F^0_{ij}dx^i\wedge dx^j,\;\;F^E=\sum_{i<j}F^E_{ij}dx^i\wedge dx^j,
\]
and
\begin{equation}
F^E_{ij}=F^0_{ij}+ \nabla^0_{x^i}A_j-\nabla^0_{x^j}A_i+ [A_i,A_j].
\label{curv}
\end{equation}
Observe that
\begin{subequations}
\begin{equation}
\begin{split}
\nabla^0_{x^i}A_j & =-\underbrace{\Bigl(\nabla^0_{x^i}\nabla^0_{x^j} C(x,y)|_{x=y} +\nabla^0_{y^i}\nabla^0_{x^j} C(x,y)|_{x=y} \,\Bigr)}_{=:T_{ij}(x)}\cdot C(x)^{-1}\\
&-\nabla^0_{x^j} C(x,y)|_{x=y}\cdot\nabla^0_{x^i}\bigl(\,C(x)^{-1}\,\bigr),
\end{split}
\label{paij}
\end{equation}
\begin{equation}
\nabla^0_{x^i}\bigl(\,C(x)^{-1}\,\bigr)=-C_x^{-1}\bigl(\,\nabla^0_{x^i} C(x)\,\bigr)C_x^{-1},
\label{naic0}
\end{equation}
\begin{equation}
\nabla^0_{x^i}C(x)=\nabla_{x^i}^0C(x,y)|_{x=y}+\nabla^0_{y^i}C(x,y)|_{x=y}.
\label{naic}
\end{equation}
\end{subequations}

\subsection{Probabilistic reconstruction of the geometry of a vector bundle}  Suppose  that we are given a smooth rank $r$ real vector bundle $E\to M$ over the smooth compact manifold    $M$. We fix   a   metric $\bsi_0$ on $E$ and a connection $\nabla^0$ on $E$ compatible with $\bsi_0$. We want to construct  a  family of sample spaces  $(\bsU_\ve,\bgamma_\ve)\subset C^\infty(E)$   with associated  special (metric, connection)-pair $(\bsi_\ve,\nabla^\ve)$ satisfying the conditions (\ref{scla},\ref{sclb},\ref{sclc}).    We use a spectral geometry  approach.

We  fix a Riemann metric $g$  on $M$ with volume  density $|dV_g|$.  We can form the covariant  Laplacian
\[
\Delta_0=\bigl(\, \nabla^0\,\bigr)^*\nabla^0: C^\infty(E)\to C^\infty(E).
\]
This is a    symmetric, nonnegative definite  second order elliptic  operator   whose principal symbol is scalar
\[
\si(\Delta_0)(\bx,\xi)= |\xi|^2_g \one_{E_\bx},\;\;\forall\bx\in M,\;\;\xi\in T^*_\bx M.
\]
Let 
\[
\spec(\Delta_0)= \lambda_1\leq \lambda_2\leq\cdots,
\]
where in the above sequence each eigenvalue  appears as many times as its  multiplicity. We fix an orthonormal eigenbasis $(\Psi_n)_{n\geq 1}$ of $L^2(E)$
\[
\Delta_0\Psi_n=\lambda_n\Psi_n,\;\;\forall n.
\]
Now fix    an even, smooth, compactly supported function $w: \bR\to[0,\infty)$. Assume that $w(0)\neq 0$. 

 For each  $\ve>0$ we have  a smoothing  selfadjoint operator
\[
W_\ve= w\bigl(\,\ve\sqrt{\Delta_0}\,\bigr): L^2(E)\to L^2(E).
\]
Define
\[
\bsU_\ve:= {\rm Range}\,(W_\ve)=\spa\bigl\{ \Psi_n;\;\;w(\ve\sqrt{\lambda_n})\neq 0\,\bigr\}\subset C^\infty(E).
\]
Note that $\bsU_\ve$  is  a finite dimensional  invariant subspace   of $W_\ve$. The restriction of $W_\ve$ to $\bsU_\ve$  is  invertible and selfadjoint with respect to the $L^2$-inner product on $\bsU_\ve$. As such,  it defines  a nondegenerate Gaussian measure $\gamma_\ve$  on $\bsU_\ve$ following the prescription (\ref{opgau})
\[
\gamma_\ve(d\bu) =\frac{1}{\sqrt{\det 2\pi W_\ve}}  e^{-\frac{1}{2}(W_\ve^{-1}\bu,\bu)_{L^2}} |d\bu|_{L^2},
\]
where $(-,-)_{L^2}$ denotes the $L^2$-inner product on $\bsU_\ve$ and $|d\bu|_{L^2}$ denotes the associated Lebesgue measure on $\bsU_\ve$. We set
\begin{equation}\label{kw}
\kappa(w): =\left(\int_0^\infty w(t) t^{m-1} dt\right){\rm vol}\,(S^{m-1})
\end{equation}
We denote generically by $L^{1,p}$ the Sobolev spaces norms  of $L^p$-functions with first order derivatives in $L^p$.

 \begin{theorem} Denote by $(\bsi_\ve,\nabla^\ve)$ the special (metric, connection)-pair determined on $E$ by the sample space $(\bsU_\ve, \gamma_\ve)$ constructed as above.  For each $\ve\geq 0$ we denote by $F^\ve$ the curvature of $\nabla^\ve$. Then for each $p\in(1,\infty)$ there exists a positive constant $K=K(p)$ such that the following hold
\[
\Vert \ve^m\bsi_\ve -\kappa(w) \bsi_0\Vert_{C^0}+\Vert\nabla^\ve-\nabla^0\Vert_{L^{1,p}}+\Vert F^\ve-F^0\Vert_{C^0}\leq K(p) \ve\;\;\mbox{as $\ve\searrow 0$}.
\]
\label{th: scl}
\end{theorem}

\begin{proof}  Consider the covariance form $C_\ve(\bx,\by)\in C^\infty(E\boxtimes E)$ determined as in Subsection \ref{ss: 42}  by  the inner product on $\bsU_\ve$  defined by the Gaussian measure $\gamma_\ve$. If we identify $E$ with $E^*$ using the metric $\bsi_0$ we can view $C_\ve$ as a section of $E\boxtimes E^*$. As such, it    coincides with the Schwartz kernel of $W_\ve$. 

The next result    contains   the key  estimates    responsible for the conclusions  in Theorem \ref{th: scl}. We defer its very technical proof to the next subsection.

\begin{lemma} Let $\rho$ denote the injectivity radius of $(M,g)$. Fix a point $\bx_0\in M$ and normal coordinates $(x^i)$ on  the open geodesic ball $B_\rho(\bx_0)$ centered at $\bx_0$.   Fix a trivialization of $E$ over $B_\rho(\bx_0)$ obtained by  $\nabla^0$-parallel transport along the geodesic rays starting at $\bx_0$. Then the following hold.

\smallskip 

\noindent (a) There exist  constants  $K,\ve_0>0$ such that
\begin{equation}
|C_\ve(x,x)-\kappa(w) \ve^{-m}\one_{E_x}| \leq K\ve^{2-m},\;\;\forall \ve<\ve_0,\;\;\mbox{ $\forall x\in B_{\rho/2}(\bx_0)$}.
\label{ker0}
\end{equation}
(b) For $1\leq i\leq m$ the limits 
\begin{equation}
\lim_{\ve\to 0}\ve^m \nabla^0_{x^i}C_\ve(x,y)_{x=y},\;\;\lim_{\ve\to 0}\ve^m\nabla^0_{y^i}C_\ve(x,y)_{x=y}
\label{ker1}
\end{equation}
exist uniformly in $x\in B_{\rho/2}(\bx_0)$  and the rate of convergence in $C^0\bigl(\,B_{\rho/2}(\bx_0)\,\bigr)$ is $O(\ve)$.  Moreover
\begin{equation}
\lim_{\ve\to 0}\ve^m \nabla^0_{x^i}C_\ve(x,y)_{x=y=\bx_0}=0.
\label{ker10}
\end{equation}
(c) For $1\leq i\neq j\leq m$ the limits 
\begin{equation}
\lim_{\ve\to 0}\ve^m \nabla^0_{x^ix^j}C_\ve(x,y)_{x=y},\;\;\lim_{\ve\to 0}\ve^m \nabla^0_{y^i}\nabla^0_{y^j}C_\ve(x,y)_{x=y}\,\;\;\lim_{\ve\to 0}\ve^m \nabla^0_{x^i}\nabla^0_{y^j}C_\ve(x,y)_{x=y}
\label{ker2}
\end{equation}
exist uniformly in $x\in B_{\rho/2}(\bx_0)$ and the rate of convergence in $C^0\bigl(\,B_{\rho/2}(\bx_0)\,\bigr)$ is $O(\ve)$.

\noindent (d) For $1\leq i\leq m$ the limit
\begin{equation}
\lim_{\ve\to 0}\ve^m\left( \,\nabla^0_{x^i}\nabla^0_{x^i} C_\ve(x,y)_{x=y}+ \nabla^0_{y^i}\nabla^0_{x^i}C_\ve(x,y)_{x=y}\,\right)
\label{ker3}
\end{equation}
exists uniformly in $x\in B_{\rho/2}(\bx_0)$ and the rate of convergence in $C^0\bigl(\,B_{\rho/2}(\bx_0)\,\bigr)$ is $O(\ve)$.
\qed
\label{lemma: key}
\end{lemma}

 Assuming  the  validity of Lemma \ref{lemma: key} we  proceed as follows. Fix  $\bx_0\in M$ and normal coordinates in $B_\rho(\bx_0)$ centered at $\bx_0$. For simplicity we write $\kappa$ instead of $\kappa(w)$. We deduce from   (\ref{ker0}) that
\[
\Vert \ve^m\bsi_\ve-\kappa \bsi_0\Vert_{C_0}=O(\ve^2)\;\;\mbox{as $\ve\to 0$}.
\]
In the sequel the Landau symbol $O$ refers to the $C^0$-norm  on  $B_{\rho/2}(\bx_0)$. Note also that  (\ref{ker0}) implies that 
\begin{equation}
C_\ve(\bx)^{-1} =\ve^m\Bigl(\kappa^{-1}\one_{E_\bx}+O(\ve^2)\,\Bigr).
\label{cinv}
\end{equation}
If we write $A^\ve:=\nabla^\ve-\nabla^0$, then    we deduce  from (\ref{a}) and (\ref{ker1})   that 
\[
A^\ve_i(x)= -\nabla^0_{x^i}C_\ve(x,y)_{x=y}\cdot C_\ve(x)^{-1}=-\ve^m\nabla^0_{x^i}C_\ve(x,y)_{x=y}\Bigl(\kappa^{-1}\one_{E_\bx}+O(\ve^2)\,\Bigr)
\]
has a  limit as $\ve\to 0$  uniform in $x\in B_{\rho/2}(\bx_0)$. We set
\begin{equation}
\bar{A}_i(x) :=\lim_{\ve\to 0}A^\ve_i(x).
\label{diff_conn}
\end{equation}
Moreover (\ref{ker10}) implies
\begin{equation}
\bar{A}_i(\bx_0) =0.
\label{diff_conn0}
\end{equation}
We have
\begin{equation}
\bigl\Vert \bar{A}_i-A^\ve_i\bigr\Vert_{C^0(B_{\rho/2}(\bx_0))}=O(\ve).
\label{diff_conna}
\end{equation} 
Using (\ref{curv})  we deduce that  along $B_\rho(\bx_0)$  and for $i\neq j$ we have
\[
F^\ve_{ij}-F^0_{ij}=\nabla^0_{x^i}A^\ve_j-\nabla^0_{x^j}A^\ve_i+ [A^\ve_i,A^\ve_j].
\]
From  (\ref{diff_conna})  we deduce
\begin{equation}
 \bigl\Vert\, [A^\ve_i,A_j^\ve]-[\bar{A}_i,\bar{A}_j]\,\bigr\Vert_{C^0(B_{\rho/2}(\bx_0))}=O(\ve).
\label{diff_conn1}
\end{equation} 
To estimate $\nabla^0_{x^i}A^\ve_j(x)$ we use (\ref{paij}) and we have
\[
\nabla^0_{x^i}A^\ve_j(x)= -T^\ve_{ij}(x) C_\ve(x)^{-1}- \nabla^0_{x^j} C_\ve(x,y)_{x=y}\cdot\nabla^0_{x^i}\bigl(\,C_\ve(x)^{-1}\,\bigr),
\]
\[
T_{ij}^\ve(x)=\nabla^0_{x^i}\nabla^0_{x^j} C_\ve(x,y)_{x=y} +\nabla^0_{y^i}\nabla^0_{x^j} C_\ve(x,y)_{x=y}.
\]
The estimate  (\ref{diff_conn}) and  Lemma \ref{lemma: key}(b) imply  that
\[
\lim_{\ve\to 0} T_{ij}^\ve(x)C_\ve(x)^{-1}
\]
exists uniformly in $x\in B_{\rho/2}(\bx_0)$ and the rate of convergence in $C^0\bigl(\,B_{\rho/2}(\bx_0)\,\bigr)$ is $O(\ve)$.  Using (\ref{naic0}), (\ref{naic}) and (\ref{diff_conn}) we deduce that 
\[
\lim_{\ve \to 0}  \nabla^0_{x^j} C_\ve(x,y)_{x=y}\cdot\nabla^0_{x^i}\bigl(\,C_\ve(x)^{-1}\,\bigr)\;\;\mbox{exists uniformly in $x\in B_{\rho/2}(\bx_0)$},
\]
and the rate of convergence in $C^0\bigl(\,B_{\rho/2}(\bx_0)\,\bigr)$ is $O(\ve)$. We conclude that
\begin{equation}
 \bar{F}_{ij}(x):=\lim_{\ve\to 0} F^\ve_{ij}(x) \;\;\mbox{exists uniformly in $x\in B_{\rho/2}(\bx_0)$},
\label{diff_curv}
\end{equation}
and 
\begin{equation}
\bigl\Vert\bar{F}_{ij}-F^\ve_{ij}\bigr\Vert_{C^0(B_{\rho/2}(\bx_0))}=O(\ve).
\label{diff_curvef}
\end{equation}
Observe now that
\[
\begin{split}
\nabla^0_{x_i}A^\ve_i(x) &= -\left(\nabla^0_{x_i}\nabla^0_{x^i}C_\ve(x,y)_{x=y}+\nabla^0_{x_i}\nabla^0_{x^i}C_\ve(x,y)_{x=y}\,\right)\cdot C(x)^{-1}\\
&-\nabla^0_{x^i}C_\ve(x,y)_{x=y}\cdot \nabla^0_{x^i}\bigl(\,C(x)^{-1}\,\bigr).
\end{split}
\]
Lemma \ref{lemma: key}(c) together with  (\ref{diff_conn}) imply that the limit
\[
\lim_{\ve\to 0} \left(\nabla^0_{x_i}\nabla^0_{x^i}C_\ve(x,y)_{x=y}+\nabla^0_{x_i}\nabla^0_{x^i}C_\ve(x,y)_{x=y}\,\right)\cdot C(x)^{-1}
\]
exists uniformly for $x\in B_{\rho/2}(\bx_0)$ and the rate of convergence in $C^0\bigl(\,B_{\rho/2}(\bx_0)\,\bigr)$ is $O(\ve)$. Finally (\ref{ker1}) and (\ref{diff_conna}) imply that
\[
\bigl\Vert\nabla^0_{x^i}C_\ve(x,y)_{x=y}\cdot \nabla^0_{x^i}\bigl(\,C(x)^{-1}\,\bigr)\bigr\Vert_{C^0(B_{\rho/2}(\bx_0))}=O(\ve).
\]
Hence
\begin{equation}
\lim_{\ve\to 0} \nabla^0_{x_i}A^\ve_i(x) \;\;\mbox{exists uniformly in $x\in B_{\rho/2}(\bx_0)$},
\label{diff_curv1}
\end{equation}
 and the rate of convergence in $C^0\bigl(\,B_{\rho/2}(\bx_0)\,\bigr)$ is $O(\ve)$. 
 
 The connection $\nabla^0$ defines a first order elliptic (Hodge) operator
\[
\eH  :\Omega^\bullet\bigl(\, \End(E)\,\bigr)\to  \Omega^\bullet\bigl(\, \End(E)\,\bigr),\;\;\eH = d^{\nabla^0}+\Bigl(d^{\nabla^0}\Bigr)^*.
\]
Since $A^\ve(x)$ converges uniformly on $B_{\rho/2}(x)$  as $\ve\to 0$, we deduce  from (\ref{diff_curv}) and (\ref{diff_curv1}) that $\eH A^\ve(x)$ converges uniformly  on $B_{\rho/2}(x)$ as $\ve\to 0$. 

 Invoking elliptic $L^p$-estimates we deduce that for any $p\in (1,\infty)$ there exists a constant $C>0$ such that for any  $\ve_1,\ve_2>0$ we have
\[
\Vert A^{\ve_1}-A^{\ve_2}\Vert_{L^{1,p}(\,B_{\rho/4}(\bx_0)\,)}\leq C\Bigl(\Vert A^{\ve_1}-A^{\ve_2}\Vert_{L^{p}(\,B_{\rho/2}(\bx_0)\,)}+\Vert \eH A^{\ve_1}-\eH A^{\ve_2}\Vert_{L^{p}(\,B_{\rho/2}(\bx_0)\,)}\Bigr).
\]
The right-hand side of the above inequality goes to $0$ as $\ve_1,\ve_2\to 0$ so
\[
\lim_{\ve_1,\ve_2\to 0} \Vert A^{\ve_1}-A^{\ve_2}\Vert_{L^{1,p}(B_{\rho/4}(\bx_0))}=0.
\]
This proves that as $\ve\to 0$ the $1$-forms  $A^\ve(x)$ converge  in the $L^{1,p}$-norm on $B_{\rho/4}(\bx_0)$. Since these forms converge uniformly to $\bar{A}$ on this ball we deduce that
\[
\lim_{\ve\to 0} \Vert A^\ve-\bar{A}\Vert_{L^{1,p}(\, B_{\rho/4}(\bx_0)\,)}=0.
\]
Since $M$ is compact we conclude that  exists a globally defined $\End(E)$-valued $1$-form 
\[
\bar{A}\in L^{1,p}\bigl(\, T^*M\otimes \End(E)\,\bigr)
\]
such that
\[
\lim_{\ve\to 0} \Vert A^\ve-\bar{A}\Vert_{L^{1,p}(M)}=0,\;\;\forall p\in (1,\infty).
\]
Moreover  the equality (\ref{ker10}) shows that $\bar{A}(\bx_0)=0$. Since the point $\bx_0$ was arbitrary we deduce $\bar{A}=0$. In turn, this implies that  $F^\ve = F^0+d^{\nabla^0} A^\ve$ converges in $L^p(M)$ to $F^0$.  From (\ref{diff_curv})   we deduce that this  convergence is in fact uniform. This proves Theorem \ref{th: scl} assuming the validity of Lemma \ref{lemma: key}.
 \end{proof}

\medskip

\subsection{Proof of Lemma \ref{lemma: key}.}    We rely on the techniques  pioneered  by L. H\"{o}rmander \cite{Hspec}  to   describe asymptotic estimates for  the Schwartz   kernel  of $W_\ve$ as $\ve\to 0$.   We follow closely the presentation in \cite[XII.2]{Tay}. We allow $w$ to be an \emph{arbitrary even  Schwartz function} $w\in\eS(\bR)$. We denote by $C^w_\ve$ the Schwartz  kernel of $w(\ve \sqrt{\Delta}_0)$. 

Fix a point $\bx_0\in M$ and normal coordinates  $(x^i)$ on $B_\rho(\bx_0)$.      We  fix a  local orthonormal  frame  $(\be_\alpha)$ of $E$ over this ball  which is $\nabla^0$-synchronous of $\bx_0$, i.e., 
\begin{equation}
\nabla^0\be_\alpha(\bx_0)=0,\;\;\forall \alpha.
\label{synchr}
\end{equation}
We   will describe another integral kernel $\eK^w_\ve(x,y)\in \Hom(E_y\otimes \bC,E_x\otimes \bC)$, defined for $x,y\in B_\rho(\bx_0)$, $|x-y|$ sufficiently small, such that
\[
C^w(x,y)=\eK^w_\ve(x,y)+O(\ve^\infty),
\]
i.e.,
\[
\Vert C^w_\ve(x,y)-\eK^w_\ve(x,y)\Vert_{C^k}=O(\ve^N)\;\;\mbox{as $\ve\to 0$},\;\;\forall  k, N\in\bZ_{>0},
\]
where the $C^k$-norm  above refers to the $C^k$-norms of functions defined  in a neighborhood of the diagonal in $M\times M$. 

Fix a smooth $a:\bR\to \bR$ such that
\[
a(t)=\begin{cases} 
0,& |t|<1,\\
1, &|t|>2.
\end{cases}
\]
For $x\in B_\rho(\bx_0)$ and $\xi\in \bR^m$ we denote by $|\xi|_x$ the length of $\xi$ as an element of $T^*_xM$.  The approximate kernel   $\eK^w_\ve(x,y)$ has the form \cite[Chap. XII, (2.2)]{Tay}
\begin{equation}
\eK^w_\ve(x,y)=\int_{\bR^m}  q_\ve(x,\xi)  e^{\ii(x-y,\xi)} d\xi,
\label{kve}
\end{equation}
where  for any positive integer $\nu$    we have
\begin{equation}
q_\ve(x,\xi) =a(|\xi|_x) w(|\xi|_x)c_0(x,\xi) +a(|\xi|_x)\sum_{j=1}^{2\nu} \ve^j w^{(j)}\bigl(\,\ve|\xi|_x\,\bigr) c_j(x,\xi) + R^\ve_\nu (\ve, x,\xi),
\label{qve}
\end{equation}
and,   for every $\ve>0$, the remainder $R^\ve_\nu(x,\xi)$ is a classical symbol of order $\leq -\nu-1$  and the  family $(R^\ve_\nu(x,\xi))_{\ve\in (0,1)}$ is bounded in the space of such symbols.

Moreover, $c_0(x,\xi)=\one_{E_x}$,  each of the  terms $c_j(x,\xi)$  is independent of $w$,  and  it has an asymptotic expansion as $\xi\to \infty$  
\[
c_j(x,\xi)\sim \sum_{k\leq \lfloor j/2\rfloor} c_{jk}(x,\xi),
\]
where $c_{jk}(x,\xi)$ is homogeneous of order $k$ in $\xi$.

\begin{slemma} Suppose that  $\phi\in \eS(\bR)$ and 
\[
c:B_\rho(\bx_0)\times (\bR^m\setminus 0 )\to  \End(E_0\otimes\bC),\;\;(x,\xi)\mapsto c(x,\xi),
\]
is a smooth function homogeneous of order $k\in\bZ$ .  We set
\begin{equation}
L_\ve[\phi,c(x)]:=\int_{\bR^m}   a(|\xi|_x) \phi(\ve|\xi|_x) c(x,\xi) d\xi,
\label{lve}
\end{equation}
\[
\hat{c}(x)=\int_{|\xi|_x=1} c(x,\xi)d\xi.
\]
Then the following hold.

\begin{enumerate}
 \item  If $k \leq -m-1$, then
 \[
 \bigl| L_\ve[\phi,c(x)]\,\bigr|=O\bigl(\,\Vert \phi\Vert_{C^0}\,\bigr).
 \]
 \item If $k=-m$, then there exist    temperate distributions
 \[
 T_{j,m}:\eS(\bR)\to \bR,\;\;j=-1,0,2,\dotsc,
 \]
 such that as $\ve\to 0$ we have  the asymptotic expansion
 \[
L_\ve[\phi,c(x)]\sim \hat{c}(x)\left((\log\ve) T_{-1,m}(\phi)+\sum_{j=0}^\infty\ve^j T_{j,m}(\phi)\right).
\]
Moreover,
\[
T_{-1,m}(\phi) = \phi(0).
\]
\item If $k>-m$, then there exist temperate distributions
\[
T_{j,k}:\eS(\bR)\to \bR,\;\;j=0,1,\dots,
\]
such that  as $\ve\to 0$ we have an asymptotic expansion
\[
L_\ve[\phi,c(x)]\sim \ve^{-m-k} \hat{c}(x)\sum_{j=0}^\infty \ve^j T_{j,m}(\phi).
\]
Moreover 
\[
T_{0,k}(\phi)= \left(\int_0^\infty\phi(s) s^{k+m-1} ds\right).
\]

\end{enumerate}
\label{slemma: asy}
\end{slemma}

\begin{proof} Part (i) is obvious because $a(|\xi|_x)c(x,\xi)$ in integrable in $\xi$ over $\bR^m$ if  the order $k$ of $c$ is  $<-m$. Assume  that $k\geq -m$. We set
\[
\hat{c}(x):=\int_{|\xi|_x=1} c(x,\xi) d\xi.
\]
We  have
\[
L_\ve[\phi,c(x)]=\int_0^\infty\left( \int_{|\xi_x|=1} c(x,t\xi_x) d\xi_x\right)a_0(t)\phi\ve(t) t^{m-1} dt.
\]
\[
=\left(\int_0^\infty a_0(t)\phi(\ve t) t^{k+m-1} dt\right)\hat{c}(x)=\ve^{-k-m}\left(\int_0^\infty a_0(s/\ve)\phi(s) s^{k+m-1} ds\right)\hat{c}(x).
\]
The last $1$-dimensional    integral has a complete     asymptotic  expansion as $\ve\to 0$ described explicitly in \cite[Eq.(4.4.22)]{BH}.    Sublemma \ref{slemma: asy} follows by unraveling  the details of this asymptotic expansion.
 \end{proof}

Fix two multi-indices  $\alpha,\beta\in\bZ^m_{\geq 0}$ such that $|\alpha|+|\beta|\leq 2$.   Using (\ref{kve})  we deduce that
\[
\pa^\alpha_x\pa^\beta_y\eK^w_\ve(x,y)|_{x=y}= (-1)^{|\beta|}\ii^{|\alpha|+|\beta|}\xi^\alpha\xi^\beta\int_{\bR^m}q(x,\xi)+\int_{\bR^m} q_1(x,\xi) d\xi
\]
where
\[
q_1(x,\xi)=\pa^\alpha_x\pa^\beta_y\Bigl(q(x,\xi)e^{\ii(x-y,\xi)}\,\Bigr)_{x=y}-  q(x,\xi)\Bigl(\pa^\alpha_x\pa^\beta_ye^{\ii(x-y,\xi)}\Bigr)_{x=y}.
\]
\[
=\sum_{0\leq \gamma<\alpha} Z_{\alpha,\beta,\gamma} \xi^\gamma\xi^\beta\pa^{\alpha-\gamma}_xq_\ve(x,y,\xi)   d\xi,
\]
and $Z_{\alpha,\beta,\gamma}$ are certain universal complex constants. Using  (\ref{qve}) with $\nu=m+2$  and  Sublemma \ref{slemma: asy} we deduce that there exist  universal  temperate distributions
\[
S^j_{\alpha,\beta}:\eS(\bR)\to \bC,\;\;j=0,1,2,
\]
and  endomorphisms
\[
\bsK^j_{\alpha,\beta}(x):E_{x}\to E_{x},\;\;j=0,1,2,
\]
depending smoothly on $  x$  but independent of $w$ such that
\begin{equation}
\ve^m\pa^\alpha_x\pa^\beta_y\eK^w_\ve(x,y)|_{x=y}=\ve^{-|\alpha|-|\beta| }\left( \sum_{j=0}^2 \ve^jS_{\alpha,\beta}^j(w)\bsK^j_{\alpha,\beta}(x)+ O(\ve^3)\,\right).
\label{asyk}
\end{equation}
Moreover, since $c_0(x,\xi)=\one_{E_x}$ we deduce
\begin{equation}
\begin{split}
S^0_{\alpha,\beta}(w)=\int_0^\infty w(t)t^{m+|\alpha|+|\beta|-1} dt,\\\
\bsK^0_{\alpha,\beta}(x)= (-1)^{|\beta|}\ii^{|\alpha|+|\beta|}\left(\int_{|\xi|=1} \xi^\alpha\xi^\beta\right)\one_{E_x}.
\end{split}
\label{kab0}
\end{equation}
For any Schwartz function $w\in \eS(\bR)$ and any $\lambda>0$ we set
\[
w_\lambda(x) = w(\lambda x).
\]
Observe that $w_\lambda(\ve\sqrt{\Delta_0})=w(\lambda\ve\sqrt{\Delta_0})$ so that, for fixed $\lambda>0$, we have
\[
\eK^{w_\lambda}_\ve=\eK^w_{\lambda\ve}+ O(\ve^\infty).
\]
Using this in (\ref{asyk}) we  deduce  that for $|\alpha|+|\beta|\leq 2$ and $j=0,1,2$ we have
\begin{equation}
S^j_{\alpha,\beta}(w_\lambda)=\lambda^{-m-|\alpha|-|\beta|+j} S^j_{\alpha,\beta}(w).
\label{homo}
\end{equation}

\begin{slemma} (a) Let  $|\alpha|+|\beta|\in \{0,2\}$. If $\phi\in \eS(\bR)$ is even,  then
\begin{equation}
S^1_{\alpha,\beta}(\phi)\bsK^1_{\alpha,\beta}(x)=0,\;\;\forall x\in B_{\rho/2}(\bx_0).
\label{odd0}
\end{equation}
(b) If $\phi\in \eS(\bR)$ is even,  then
\begin{equation}
\lim_{\ve\to }\ve^m \nabla^0_{x^i}\eK_\ve^\phi(x,y)|_{x=y=\bx_0}=0.
\label{odd1}
\end{equation}
\label{slemma2}
\end{slemma}

\begin{proof} Denote by $\eS_+(\bR)$ the space of even Schwartz functions on $\bR$ and by $\eX_{\alpha,\beta}$ the subspace  of   of $\eS_+(\bR)$ consisting of  functions $\phi$ satisfying (\ref{odd0}).   Clearly $\eX_{\alpha,\beta}$ is a closed subspace of $\eS_+$ so it suffices to prove that  $\eX_{\alpha,\beta}$ is  dense  in $\eS_+(\bR)$ with respect to the natural locally convex topology of $\eS(\bR)$.  The  family $\gamma_\lambda(s)=e^{-\lambda^2 s^2}$ spans a vector space dense in $\eS_+(\bR)$;  see \cite[Chap. 8, Lemma 2.3]{Tay2}. Thus, it suffices to show that $\gamma_\lambda\in \eX_{\alpha,\beta}$ for any $\lambda>0$. In view of the homogeneity condition   (\ref{homo}) we see that 
\[
\gamma_1\in \eX_{\alpha,\beta}\iff \gamma_\lambda\in \eX_{\alpha,\beta},\;\;\forall\lambda>0.  
\]
For $t>0$ we denote by $H_t$ the   heat kernel, i.e., the Schwartz kernel of $e^{-t\Delta_0}$.  Note that $H_{\ve^2}$ is the the Schwartz kernel of $\gamma_1(\ve\sqrt{\Delta_0})$.

The  heat kernel $H_t(x,y)$ has a rather well understood   structure.  We denote by $d(x,y)$ the geodesic distance between $x,y\in B_{\rho/2}(\bx_0)$ with respect to the metric  $g$ on $M$.    For $x,y$ in a neighborhood of the diagonal we have an asymptotic expansion  as $t\searrow 0$ (see  \cite[Thm. 7.15]{Roe})
\begin{equation}
H_t(x,y)=h_t(x,y)\underbrace{\sum_{\nu=0}^\infty t^\nu\Theta_\nu(x,y)}_{=:\Theta_t(x,y)},\;\;\nu\in \bZ_{\geq 0},
\label{heat}
\end{equation} 
where $\Theta_k(x,y)\in \Hom(E_y,E_x)$  and
\[
h_t(x,y)=t^{-\frac{m}{2}}   e^{-\frac{d(x,y)^2}{4t}}.
\]
The  asymptotic  expansion (\ref{heat}) is differentiable with respect to all the variables $t,x,y$. Hence
\begin{equation}
\ve^mH_{\ve^2}(x,y) \sim e^{-u_\ve}\sum_{\nu=0}^\infty \ve^{2\nu}\Theta_\nu(x,y),
\label{heatve}
\end{equation}
where $u_\ve:=\frac{d(x,y)^2}{4\ve^2}$. When $x=y$ we have $u_\ve=0$ and thus
\[
\ve^m H_{\ve^2} (x,x)\sim  \sum_{\nu=0}^\infty \ve^{2\nu}\Theta_\nu(x,x).
\]
This proves  (\ref{odd0}) in the case $\alpha=\beta=0$ for the test function $\gamma_1$ since the  expansion in the right-hand side  above involves only even powers of $\ve$.

Differentiating (\ref{heatve}) we deduce
\begin{equation}
\ve^m\nabla^0_{x^i}H_{\ve^2} (x,y)  \sim -(\pa_{x^i}u_\ve) e^{-u_\ve} \sum_{\nu=0}^\infty \ve^{2\nu}\Theta_\nu(x,y)+e^{-u_\ve}\sum_{\nu=0}^\infty \ve^{2\nu}\nabla^0_{x^i}\Theta_\nu(x,y) .
\label{heat3}
\end{equation}
To compute  $\ve^m\nabla^0_{x^j}\nabla^0_{x^i}H_{\ve^2}(x,y)$ when $x=y$ we will  take into account that $\pa_{x^i}u_\ve=0$ when $x=y$. We deduce
\begin{equation}
\begin{split}
\ve^m\nabla^0_{x^j}\nabla^0_{x^i}H_{\ve^2}(x,y)_{x=y}&\sim \frac{1}{4\ve^2}\pa^2_{x^jx^i}d(x,y)^2|_{x=y} \sum_{\nu=0}^\infty \ve^{2\nu}\Theta_\nu(x,x)\\
&+ \sum_{\nu=0}^\infty \ve^{2\nu}\nabla^0_{x^j}\nabla^0_{x^i}\Theta_\nu(x,y)_{x=y}.
\end{split}
\label{heat4}
\end{equation}
This proves that  $\ve^{m+2}\nabla^0_{x^j}\nabla^0_{x^i}H_{\ve^2}(x,y)_{x=y}$ has an asymptotic expansion in \emph{even, nonnegative powers of $\ve$}.  Arguing in a similar fashion we deduce that   the kernels 
\[
\ve^{m+2}\nabla^0_{y^j}\nabla^0_{y^i}H_{\ve^2}(x,y)_{x=y},\;\; \ve^{m+2}\nabla^0_{y^j}\nabla^0_{x^i}H_{\ve^2}(x,y)_{x=y}
\]
also   have asymptotic expansions in \emph{even, nonnegative powers of $\ve$}.  We conclude  that $\gamma_1\in \eX_{\alpha,\beta}$  if $|\alpha|+|\beta|=2$.

Let us observe that (\ref{heat3}) implies
\[
\ve^m\nabla^0_{x^i}H_{\ve^2}(x,y)|_{x=y}\sim \sum_{\nu=0}^\infty \ve^{2\nu}\nabla^0_{x^i}\Theta_\nu(x,y)_{x=y}.
\]
We deduce that 
\[
\lim_{\ve\to 0}\ve^m\nabla^0_{x^i}H_{\ve^2}(x,y)|_{x=y}=\nabla^0_{x^i}\Theta_0(x,y)|_{x=y}.
\]
From the transport equations \cite[Eq.(7.17)]{Roe}    we  deduce that, \emph{in normal coordinates at $\bx_0$},  and under the synchronicity  condition (\ref{synchr}), we have
\[
\nabla^0_{x^i}\Theta_0(x,y)|_{x=y=\bx_0}=0.
\]
This proves  (\ref{odd1}) for $\phi=\gamma_1$ and thus for any  even Schwartz  function $\phi$.
\end{proof}

We can now complete the proof of Lemma \ref{lemma: key}.  Using  (\ref{asyk}) and (\ref{kab0}) with $\alpha=\beta=0$ and Sublemma \ref{slemma2}(a) we deduce that
\[
\ve^mC_\ve(x,x)=\kappa(w)\one_{E_x} +O(\ve^2),
\]
where  we recall that 
\[
\kappa(w)=\left(\int_0^\infty w(t) t^{m-1} dt\right){\rm vol}\,(S^{m-1}).
\]
For $1\leq i\leq m$ we set
\[
\alpha_i=(\delta_{i1},\dotsc,\delta_{im})\in\bZ^m_{\geq 0},
\]
where $\delta_{ij}$ is Kronecker's delta.  From (\ref{kab0}) we deduce that
\[
\bsK^0_{\alpha_j,0}=-\bsK^0_{0,\alpha_j}=\ii\left(\int_{|\xi|=1}\xi_j\right)\one_{E_x}=0.
\]
Thus
\[
\ve^m\nabla^0_{x^i}C_\ve(x,y)_{x=y}= S^1_{\alpha_i,0}(w)\bsK^1_{\alpha_i,0}+O(\ve),
\]
\[
\ve^m\nabla^0_{y^i}C_\ve(x,y)_{x=y}= S^1_{\alpha_i,0}(w)\bsK^1_{0,\alpha_i}+O(\ve).
\]
These  estimates prove (\ref{ker1}). The equality (\ref{ker10}) follows from (\ref{odd1}).

 From  (\ref{kab0})  we deduce that for $1\leq i\neq j\leq m$ 
\[
\bsK^0_{\alpha_i+\alpha_j,0} (x)=-\bsK^0_{\alpha_i,\alpha_j} (x)=\ii \left(\int_{|\xi|=1}\xi_i\xi_j\right)\one_{E_x}=0,
\] 
and invoking (\ref{odd0}) we conclude that
\[
\ve^m\nabla^0_{x^i}\nabla^0_{x^j}C_\ve(x,y)_{x=y} = S^2_{\alpha_i+\alpha_j,0}(w)\bsK^2_{\alpha_i+\alpha_j,0}(x) +O(\ve),
\]
\[
\ve^m\nabla^0_{x^i}\nabla^0_{y^j}C_\ve(x,y)_{x=y} = S^2_{\alpha_i,\alpha_j}(w)\bsK^2_{\alpha_i,\alpha_j}(x)+O(\ve),
\]
\[
\ve^m\nabla^0_{y^i}\nabla^0_{y^j}C_\ve(x,y)_{x=y} = S^2_{0,\alpha_i+\alpha_j}(w)\bsK^2_{0,\alpha_i+\alpha_j}(x) +O(\ve).
\]
These estimates prove (\ref{ker2}). Note that Sublemma \ref{slemma2}  implies that
\[
\ve^m\left( \,\nabla^0_{x^i}\nabla^0_{x^i} C_\ve(x,y)_{x=y}+ \nabla^0_{y^i}\nabla^0_{x^i}C_\ve(x,y)_{x=y}\,\right)
\]
\[
=\ve^{-2}\Bigl(S^0_{2\alpha_i,0}(w)\bsK^0_{2\alpha_i,0}(x)+S^0_{\alpha_i,\alpha_i}(w)\bsK^0_{0,2\alpha_i}(x)\Bigr) 
\]
\[
+\Bigl(S^2_{2\alpha_i,0}(w)\bsK^2_{2\alpha_i,0}(x)+S^2_{\alpha_i,\alpha_i}(w)\bsK^2_{\alpha_i,\alpha_i}(x)\Bigr) +O(\ve).
\]
The equalities (\ref{kab0}) imply that
\[
S^0_{2\alpha_i,0}(w)\bsK^0_{2\alpha_i,0}(x)+S^0_{\alpha_i,\alpha_i}(w)\bsK^0_{\alpha_i,\alpha_i}(x)=0.
\]
This proves  (\ref{ker3})  and completes the proof of Lemma \ref{lemma: key}.\qed

\end{document}